\documentclass[preprint,12pt]{elsarticle}

\usepackage[T1]{fontenc}
\usepackage[utf8]{inputenc}
\usepackage{lmodern}
\usepackage{amsmath,amssymb,amsthm,mathtools}
\usepackage{microtype}
\usepackage[left=0.95in,right=0.95in,top=0.95in,bottom=0.95in]{geometry}
\usepackage{enumitem}
\usepackage[hidelinks]{hyperref}

\journal{Nonlinear Analysis}
\numberwithin{equation}{section}
\allowdisplaybreaks[4]
\newtheorem{theorem}{Theorem}[section]
\newtheorem{proposition}[theorem]{Proposition}
\newtheorem{lemma}[theorem]{Lemma}
\theoremstyle{remark}
\newtheorem{remark}[theorem]{Remark}

\def\R{{\mathbb R}}
\newcommand{\C}{\mathbb{C}}
\newcommand{\Sph}{\mathbb{S}}
\newcommand{\dotH}{\dot H}
\newcommand{\eps}{\varepsilon}
\newcommand{\Awindow}{\mathcal A}
\newcommand{\I}{\mathrm{i}}
\newcommand{\norm}[1]{\lVert #1\rVert}
\newcommand{\ip}[2]{\langle #1,#2\rangle}
\newcommand{\hpair}[2]{(#1\mid #2)}

\begin{document}

\begin{frontmatter}

\title{Large-Data Global Criteria for Two-Dimensional Critical-Energy Schrödinger Maps}

\author[aff1]{Pang-Hung Chung\fnref{fund1}}
\ead{penghongzhong@yahoo.com}

\author[aff2]{Dan Han\corref{cor1}}
\ead{dan.han@louisville.edu}

\cortext[cor1]{Corresponding author.}
\fntext[fund1]{The first author was supported by the Project ``Research on Nonlinear Partial Differential Equations'' (No. 2024KYCXTD018), Special Projects in Key Areas of Guangdong Province (No. ZDZX1088), and the Fund for Science and Technology of Guangzhou (No. 202102080428).}

\affiliation[aff1]{organization={Department of Applied Mathematics, Guangdong University of Education},
                  city={Guangzhou},
                  postcode={510640},
                  country={P. R. China}}

\affiliation[aff2]{organization={Department of Mathematics, University of Louisville},
                  city={Louisville},
                  state={KY},
                  postcode={40245},
                  country={USA}}

\begin{abstract}
We study large-data continuation below the critical energy for the two-dimensional sphere-valued Schrödinger map
\[
\partial_t u=u\times\Delta u,
\qquad u:[0,T_+)\times\R^2\to\Sph^2,
\qquad E(u_0)<4\pi.
\]
Working directly with the original equation through the complex tangential fields $W_m=\partial_m u+\I u\times\partial_m u$ and the ambient connections $\Gamma_\alpha$, we prove frequency-envelope bounds governed by the critical norm $\|\nabla u\|_{\ell^2L^4_{t,x}}$. As consequences, $T_+=\infty$ if $\|\nabla u\|_{\ell^2L^4([0,T_+)\times\R^2)}<\infty$, and also if a terminal one-window high-frequency tail $\Awindow(t)$ satisfies $\liminf_{t\uparrow T_+}\Awindow(t)<\eps_*$. Thus any finite-time failure must force persistent high-frequency concentration on all sufficiently late critical parabolic windows.
\end{abstract}


\end{frontmatter}

\section{Introduction}

We consider the two-dimensional sphere-valued Schrödinger map equation
\begin{equation}\label{eq:SM-en}
\partial_t u = u \times \Delta u,
\qquad
u:I\times\R^2\to \Sph^2,
\end{equation}
whose conserved energy is
\begin{equation}\label{eq:energy-en}
E(u(t)) := \frac12\int_{\R^2} |\nabla u(t,x)|^2\,dx.
\end{equation}
Equation \eqref{eq:SM-en} and the energy \eqref{eq:energy-en} are invariant under the scaling
\[
u(t,x)\mapsto u(\lambda^2 t,\lambda x),
\]
and therefore the problem is energy-critical in two dimensions. In contrast with general semilinear dispersive equations, there are two simultaneous layers of criticality here. On the one hand, $\dotH^1_x$ is the scaling-critical regularity. On the other hand, the geometric constraint $u\in\Sph^2$ fixes the nonlinearity into a derivative, quasilinear dispersive structure. Consequently, any continuation theory must close the highest-order energy estimates without losing one spatial derivative.

In the two-dimensional $\Sph^2$-valued case, the number $4\pi$ is not an accidental technical threshold; it is the energy of the lowest nontrivial harmonic sphere. Hence, in the trivial homotopy class, the condition $E(u_0)<4\pi$ rules out bubble formation, which is the most basic geometric mechanism of failure. At the same time, the heat flow issued from each time slice converges to a constant state, and this is precisely the geometric reason why the caloric gauge can be organized globally in the subthreshold regime; see \cite{TaoCaloric,BIKd4,BIKT,SmithGeom,SmithCond}. In several settings with additional structure, one already has subthreshold or stronger global results, for instance in equivariant and near-harmonic regimes \cite{EquivariantDuke,MerleRaphaelRodnianski,GustafsonKangTsaiCPAM,GustafsonKangTsaiDuke}. For general large data without symmetry, however, the decisive issue for the global theory remains the control of a certain critical spacetime norm.

A precise formulation of this critical norm is
\begin{equation}\label{eq:def-control-en}
\|f\|_{\ell^2L^4_{t,x}(I\times\R^2)}
:=
\Biggl(
\sum_{k\in\mathbb Z}
\|P_k f\|_{L^4_{t,x}(I\times\R^2)}^2
\Biggr)^{1/2},
\end{equation}
where $P_k$ denotes the standard Littlewood--Paley projection to spatial frequencies $|\xi|\sim 2^k$; for instance one may take $P_k=\varphi(2^{-k}|D|)$ with a smooth annular cutoff $\varphi$. This is the controlling norm emphasized by Dodson--Smith \cite{DodsonSmith}. Structurally, the reason why this norm is natural is not merely that $L^4_{t,x}$ is scaling-critical for the Schrödinger flow. More importantly, in any differentiated form of the equation, the most delicate nonlinearity is always a first-order transport-type term. If one regards $W_x$ as the complex tangential representation of $\nabla u$, and the connection as a quadratic quantity obtained by integration along the heat flow, then the nonperturbative part may be schematically written as
\[
(\nabla^{-1}(W_x\overline{W_x}))\,\nabla W_x.
\]
After dyadic frequency decomposition, what lies exactly at the critical threshold is the bilinear interaction
\[
L^4_{t,x}\times L^4_{t,x}\to L^2_{t,x},
\]
and this is exactly reflected by the $\ell^2$-summation structure in \eqref{eq:def-control-en}. In other words, $\ell^2L^4$ is not merely a convenient posterior notation; it is the norm forced by the frequency geometry of the critical transport coupling.

Existing continuation theories often first rewrite the differentiated fields, under a suitable gauge, as a scalar covariant Schrödinger system, and then establish linear and bilinear estimates for that system \cite{BIKT,SmithCond,DodsonSmith}. Our point of departure is different: we would like to reconstruct this chain of reasoning as much as possible directly on the original equation and its pullback tangent bundle, without taking the scalar gauge system as a logical input. To this end, in the next section we introduce the complex tangential fields
\[
W_m:=\partial_m u+\mathrm{i}\,u\times\partial_m u,
\qquad m=1,2,
\]
and the ambient connection matrices
\[
\Gamma_\alpha Z:=u\,\langle \partial_\alpha u,Z\rangle,
\qquad Z\in \mathbb C^3.
\]
With these notations, the original equation can be written directly in the free Schrödinger form
\begin{equation}\label{eq:intro-direct-sch-en}
(\mathrm{i}\partial_t+\Delta)W_m
=
-2\Gamma_j\partial_jW_m-\Gamma_j^2W_m+\mathcal Q_m(W,\overline W,W),
\end{equation}
and the corresponding heat-flow slices satisfy
\begin{equation}\label{eq:intro-direct-heat-en}
(\partial_s-\Delta)W_m
=
2\Gamma_j\partial_jW_m+\mathcal V_m.
\end{equation}
The precise definitions of $W_m$ and $\Gamma_\alpha$ will be given at the end of this section; for the moment, it suffices to regard them as the complex tangential fields and ambient connection matrices which make the dominant quasilinear structure of the original equation explicit. This already shows that, even without first compressing the problem into a scalar gauge system, the genuine difficulty is still concentrated in the same nonperturbative transport term $2\Gamma_j\partial_jW_m$. Thus, the core continuation mechanism itself does not change; what changes is that it now has to be made explicit and closed directly at the level of the original equation.

The solutions treated here are first of all the maximal-lifespan local solutions provided by the classical local theory. More precisely, when
\[
u(0)-Q\in L_x^2,
\qquad
u(0)\in \dotH^1\cap\dotH^3,
\qquad
E(u_0)<4\pi,
\]
classical results yield a unique maximal-lifespan solution
\[
u\in L^\infty_{t,\mathrm{loc}}(I_{\max};\dotH^1\cap\dotH^3),
\qquad
I_{\max}=[0,T_+),
\]
see \cite{SulemSulemBardos,ChangShatahUhlenbeck,NahmodStefanovUhlenbeck,IK,McGahagan,NahmodShatahVegaZeng}. We do not re-establish this local theory here. Instead, we ask: what critical spacetime control is sufficient to prevent $T_+$ from being finite? And what type of local smallness near the terminal time is enough to rule out finite-time failure?

The strategy of the paper is to keep this continuation question tied to the original geometric equation. After deriving the bundle-valued Schrödinger and heat-flow equations directly from \eqref{eq:SM-en}, we build the dyadic resolution spaces and frequency envelopes in terms of the complex tangential fields. The comparison of the controlling norm, the free propagation estimates, and the frequency-localized product estimates are then proved in the text before they are used in the bootstrap. The bilinear interaction that is needed later is also written as a two-body equation for products of dyadic heat-flow slices. Thus the proof begins with the original equation, passes to the complex tangential-field formulation, establishes dyadic heat-flow bounds for those fields, and finally uses these bounds to obtain the continuation criterion. This organization is intended to make transparent which part of the argument is geometric and which part is dispersive.

The quantitative heart of the argument is a frequency-envelope estimate: on intervals where the controlling norm is small, each $X_k$-mode of $W_x$ is controlled by the initial frequency envelope together with the $\ell^2L^4$ contribution on the same interval. This estimate is proved later inside the bootstrap argument and then converted into a local continuation proposition. We record in the introduction the two global consequences that will be used as the main conclusions of the paper: the first globalizes the solution from the finiteness of the controlling norm over the whole lifespan, and the second localizes the same mechanism to a single critical parabolic window before a possible terminal time.

\begin{theorem}[Globalization from the controlling norm]\label{thm:globalize-en}
Let $u$ be the maximal-lifespan solution to \eqref{eq:SM-en} with initial time $t=0$, and write its lifespan interval as $I_{\max}=[0,T_+)$. Assume
\[
u(0)-Q\in L_x^2,
\qquad
u(0)\in \dotH^1\cap\dotH^3,
\qquad
E(u_0)<4\pi.
\]
If
\begin{equation}\label{eq:forward-finite-en}
\|\nabla u\|_{\ell^2L^4_{t,x}([0,T_+)\times\R^2)}<\infty,
\end{equation}
then
\[
T_+=+\infty,
\qquad
I_{\max}=[0,+\infty),
\qquad
u\in L^\infty_{t,\mathrm{loc}}\bigl([0,+\infty);\dotH^1\cap\dotH^3\bigr).
\]
In particular, for every $0<T<\infty$,
\[
\|\nabla u\|_{\ell^2L^4_{t,x}([0,T]\times\R^2)}<\infty.
\]
\end{theorem}

The second main theorem comes from the same local continuation mechanism, but uses more localized information. For $t>0$ and an integer $k$, set
\[
J_{t,k}:=[t-2^{-2k},t],
\]
and by convention we only include those $k$ for which $J_{t,k}\subset[0,T_+)$. Define the terminal one-window high-frequency-tail function by
\begin{equation}\label{eq:A-window-en}
\Awindow(t):=
\inf_{k:\,J_{t,k}\subset[0,T_+)}
\Biggl(
\sum_{j\ge k-C_0}
\|P_j\nabla u\|_{L^4_{t,x}(J_{t,k}\times\R^2)}^2
\Biggr)^{1/2}.
\end{equation}
Here $C_0$ will be fixed later by the one-window continuation proposition. The quantity in \eqref{eq:A-window-en} measures whether, among all critical parabolic windows ending at time $t$, there is at least one whose high-frequency tail is already sufficiently small.

\begin{theorem}[One-window globalization]\label{thm:window-global-en}
Fix $E_0<4\pi$. Then there exist constants $C_0\ge 20$, $\eps_*=\eps_*(E_0)>0$, and $T_*=T_*(E_0)>0$ such that the following holds. Let $u$ be a maximal-lifespan solution to \eqref{eq:SM-en} with $I_{\max}=[0,T_+)$, and assume
\[
u(0)-Q\in L_x^2,
\qquad
u(0)\in \dotH^1\cap\dotH^3,
\qquad
E(u_0)\le E_0.
\]
If, in the sense of a finite terminal time,
\begin{equation}\label{eq:window-global-ass-en}
\liminf_{t\uparrow T_+}\Awindow(t)<\eps_*,
\end{equation}
then $T_+=+\infty$. Equivalently, if $T_+<\infty$, then there exists $\rho=\rho(E_0)>0$ such that for every $t\in(T_+-\rho,T_+)$ and every $k$ satisfying $J_{t,k}\subset[0,T_+)$,
\begin{equation}\label{eq:late-window-lower-en}
\Biggl(
\sum_{j\ge k-C_0}
\|P_j\nabla u\|_{L^4_{t,x}(J_{t,k}\times\R^2)}^2
\Biggr)^{1/2}
\ge \eps_*.
\end{equation}
\end{theorem}

The logical relation between these two theorems is hierarchical. Theorem \ref{thm:globalize-en} says that if the local continuation mechanism can be triggered repeatedly near the end of the maximal lifespan by the small tail of the global controlling norm, then a finite terminal time cannot occur. Theorem \ref{thm:window-global-en} says further that the information needed to trigger local continuation can in fact be compressed into the high-frequency tail on a single dyadic parabolic window before the terminal time. Both rely on the frequency-envelope proposition proved later: that proposition is the quantitative estimate that truly has to be closed in the paper, rather than an independent main statement in the introduction. In this sense, the hierarchy of results in the present work may be summarized as follows: first establish frequency-envelope control directly at the level of the original equation, then derive local continuation from it, and finally obtain globalization from the controlling norm and one-window globalization.

The real analytical difficulties in the proof are threefold. First, if one does not first rewrite the problem as a scalar gauge system, then $W_m$ still takes values in $\mathbb C^3$ as a tangential field, and the nonlinearity $\Gamma_j\partial_jW_m$ is a genuinely matrix-valued first-order transport term rather than a simple semilinear perturbation. Thus any continuation argument must deal directly with the coupling between matrix connections and highest-order derivatives. Second, the continuation criteria genuinely depend on the $\ell^2L^4$ controlling norm, and we need to prove directly at the level of the original equation that it is equivalent to the dyadic control norm of the complex tangential field, while making the key comparison estimate, local smoothing estimate, and frequency-localized product estimate as self-contained as possible in the body of the paper. Third, the bilinear $L^2_{t,x}$ mechanism does not automatically appear as a scalar product at the level of the original equation. It must therefore be rewritten as an explicit two-body interaction proposition in order to close the bootstrap without appealing to a ready-made bundle-valued black box.

From the heat-flow point of view, these difficulties all point to the same critical scale. For fixed heat time $s>0$, the heat-flow slices satisfy estimates with the parabolic decay factor
\[
\|P_kW_x(s)\|_{L_t^\infty L_x^2}
+
\|P_kW_x(s)\|_{L^4_{t,x}}
\lesssim
(1+s2^{2k})^{-4}\,\upsilon_k,
\]
where $\upsilon_k$ is the frequency envelope generated by the initial data and the controlling norm. This factor shows that the dominant interaction at frequency $2^k$ can only leave a visible contribution near the heat scale $s\approx 2^{-2k}$. At the same time, the natural length of the corresponding time window is also $2^{-2k}$. Thus, the one-window globalization criterion in Theorem \ref{thm:window-global-en} is not an additional design feature; it is determined by the matching of free dispersion, heat-flow smoothing, and the transport term $2\Gamma_j\partial_jW_m$ at the same critical scale.

The argument developed here has three main features. First, we rewrite the differentiated critical structure directly in the original-equation form \eqref{eq:intro-direct-sch-en}--\eqref{eq:intro-direct-heat-en}, thereby keeping the geometric information inside the pullback tangent bundle and the ambient connection matrices, rather than compressing it in advance into scalar unknowns in a special gauge. Second, the comparison for the controlling norm, the free linear estimates, and the frequency-localized product estimates in Section 2 are all reproved in the body of the paper. In particular, the lateral Strichartz estimates and the two-dimensional Galilean maximal-function endpoint are also derived from local smoothing, maximal cross-section bounds, and interpolation. Third, Section 3 promotes the key bilinear mechanism into an explicit inhomogeneous two-body interaction proposition. This allows the subsequent bootstrap argument to proceed directly along the chain ``first-order transport term + heat-flow decay + two-body interaction,'' and ultimately yields the one-window globalization mechanism of Theorem \ref{thm:window-global-en}.

This paper does not claim to establish unconditional global well-posedness for general subthreshold large data. What we prove are two global solvability criteria. Their significance is that they reduce the large-data global problem to the control of either the controlling norm or the high-frequency tail on a single pre-terminal window. Once such control is supplied by some additional structure---for example, dispersed energy, a small critical Besov norm, or extra symmetry---Theorems \ref{thm:globalize-en} and \ref{thm:window-global-en} immediately convert that information into global existence. In this sense, the purpose of the paper is not to replace existing subthreshold global theories, but to extract the continuation mechanism in the passage from local to global theory and to make it as explicit as possible directly at the level of the original equation.

The remainder of the paper consists of three sections. Section 2 develops the geometric equations on the pullback tangent bundle, the function spaces, and the frequency envelopes, and formulates the main frequency-envelope estimate as a bootstrap proposition together with the local continuation statement needed later. Section 3 proves heat-flow decay, two-body interaction estimates, and perturbative bounds. Section 4 treats the nonperturbative transport term, closes the bootstrap, first proves the main frequency-envelope proposition, and then derives Theorems \ref{thm:globalize-en} and \ref{thm:window-global-en}; the terminal one-window continuation and ultraviolet-concentration statement will appear there as a proposition.

\medskip
\noindent\textbf{Notation and conventions.}
To avoid repetition later, we now fix several groups of frequently used notations. Throughout the paper we fix a constant state $Q\in\Sph^2$, and write $P_k$ for the standard Littlewood--Paley projections. For a tangential vector field $X$ along $u$, we set
\begin{equation}\label{eq:intro-basic-notation-en}
\Pi_uX:=X-\langle X,u\rangle u,
\qquad
J_uX:=u\times X,
\qquad
\nabla_\alpha X:=\Pi_u(\partial_\alpha X),
\qquad
\psi_m:=\partial_m u.
\end{equation}
We also define
\begin{equation}\label{eq:intro-complex-notation-en}
W_m:=\psi_m+\mathrm{i}J_u\psi_m,
\qquad
\Gamma_\alpha:=u\otimes(\partial_\alpha u)^\top-(\partial_\alpha u)\otimes u^\top.
\end{equation}
On $\mathbb C^3$, $\langle\cdot,\cdot\rangle$ always denotes the complex bilinear pairing, while $(\cdot\mid\cdot)$ denotes the Hermitian pairing. The rest of the paper will uniformly use \eqref{eq:intro-basic-notation-en}--\eqref{eq:intro-complex-notation-en}, and these notations will not be redefined.

\bigskip

\section{Original-Equation Framework}

In this section we assemble, in a directly usable form, the original-equation structure that will actually be used in the subsequent frequency analysis. In order, we first record the basic geometric identities for sphere-valued maps on the pullback tangent bundle, then derive the covariant differential system, and finally rewrite it in terms of the complex tangential fields $W_m$ and the ambient connection matrices $\Gamma_\alpha$. After this reformulation, the quasilinear part of the original equation will appear explicitly as the free Schrödinger operator together with a matrix-valued first-order transport term. All estimates in Sections 3 and 4 are, in essence, organized around this dominant structure.

\subsection{Geometric equations}

The notation has already been fixed in \eqref{eq:intro-basic-notation-en}--\eqref{eq:intro-complex-notation-en}. In this subsection we collect several basic relations that will be used repeatedly later. If $X\in T_u\Sph$, then
\[
J_u^2X=-X,
\qquad
\ip{J_uX}{J_uY}=\ip{X}{Y},
\qquad
\ip{J_uX}{Y}=-\ip{X}{J_uY}.
\]
Moreover, the tension field is
\[
\tau(u):=\nabla_j\psi_j
=\Delta u+|\nabla u|^2u,
\qquad
\partial_tu = J_u\tau(u).
\]
The connection is torsion-free, and therefore
\begin{equation}\label{eq:torsion-en}
\nabla_j\psi_k=\nabla_k\psi_j.
\end{equation}
The curvature tensor of the sphere satisfies
\begin{equation}\label{eq:curvature-en}
R(X,Y)Z = \ip{Y}{Z}X-\ip{X}{Z}Y.
\end{equation}

\begin{lemma}[Parallel complex structure]\label{lem:Jparallel-en}
For any tangential vector field $X$ along $u$, one has
\[
\nabla_\alpha(J_uX)=J_u(\nabla_\alpha X).
\]
\end{lemma}

\begin{proof}
Differentiate $u\times X$ and project onto the tangent space:
\[
\nabla_\alpha(J_uX)
=\Pi_u\bigl(\partial_\alpha u\times X + u\times \partial_\alpha X\bigr).
\]
Here $\partial_\alpha u\times X$ is orthogonal to $T_u\Sph$, and so its tangential projection vanishes. Hence
\[
\nabla_\alpha(J_uX)=u\times \Pi_u(\partial_\alpha X)=J_u\nabla_\alpha X.
\qedhere
\]
\end{proof}

The geometric identities from the preceding paragraph now lift the original equation directly to a covariant differential system on the pullback tangent bundle. Both the time evolution and the heat-flow evolution involve the covariant Laplacian together with cubic curvature terms; in particular, the highest-order derivatives appear in a geometrically natural form. This gives the following system.

\begin{proposition}[Covariant equations]\label{prop:covdiff-en}
Let $u$ be a smooth solution of \eqref{eq:SM-en}. Then
\begin{equation}\label{eq:cov-sch-psi-en}
\nabla_t\psi_m = J_u\nabla_j\nabla_j\psi_m + R(\psi_m,\psi_j)J_u\psi_j.
\end{equation}
If, for each fixed $t$, we denote by $u(s,t,x)$ the harmonic map heat flow with initial data $u(t,\cdot)$, then
\begin{equation}\label{eq:cov-heat-psi-en}
\nabla_s\psi_m = \nabla_j\nabla_j\psi_m + R(\psi_m,\psi_j)\psi_j.
\end{equation}
\end{proposition}

\begin{proof}
For \eqref{eq:cov-sch-psi-en}, by the original equation and Lemma \ref{lem:Jparallel-en},
\[
\nabla_t\psi_m = \nabla_m\partial_tu
= \nabla_m(J_u\nabla_j\psi_j)
= J_u\nabla_m\nabla_j\psi_j.
\]
Now commute $\nabla_m$ and $\nabla_j$:
\[
\nabla_m\nabla_j\psi_j
=\nabla_j\nabla_m\psi_j+R(\psi_m,\psi_j)\psi_j.
\]
By the torsion-free identity \eqref{eq:torsion-en},
\[
\nabla_j\nabla_m\psi_j=\nabla_j\nabla_j\psi_m,
\]
and this yields \eqref{eq:cov-sch-psi-en}.

For \eqref{eq:cov-heat-psi-en}, one only needs to substitute $\partial_su=\tau(u)=\nabla_j\psi_j$:
\[
\nabla_s\psi_m
=\nabla_m\nabla_j\psi_j
=\nabla_j\nabla_m\psi_j + R(\psi_m,\psi_j)\psi_j
=\nabla_j\nabla_j\psi_m + R(\psi_m,\psi_j)\psi_j.
\qedhere
\]
\end{proof}

We next pass from real tangential fields to the complex fields fixed in \eqref{eq:intro-complex-notation-en}. The field $W_m$ packages $\psi_m$ and $J_u\psi_m$ into a single complex tangential object; this allows the curvature terms, the energy density, and the transport structure appearing later to be written in a compact complex bilinear form. First, $W_m$ is a tangential $(1,0)$-type field satisfying
\begin{equation}\label{eq:JW-en}
J_uW_m = -\I W_m.
\end{equation}
Moreover, one has the explicit algebraic relations
\begin{equation}\label{eq:WW-en}
\hpair{W_m}{W_m}=2|\psi_m|^2,
\qquad
\Re\hpair{W_m}{W_j}=2\ip{\psi_m}{\psi_j},
\qquad
\Im\hpair{W_m}{W_j}=-2\ip{\psi_m}{J_u\psi_j}.
\end{equation}
The first identity gives the Hermitian length, while the latter two show that the real and imaginary parts in the $W$-variables correspond respectively to the Euclidean pairing between real tangential fields and the $J_u$-twisted pairing. Hence once one passes to $W_m$, the energy density, the mass flux, and the curvature cubic terms can all be uniformly written in terms of $\hpair{\cdot}{\cdot}$.

\begin{proposition}[Complex-field equations]\label{prop:cov-W-en}
The fields $W_m$ satisfy
\begin{equation}\label{eq:cov-sch-W-en}
(\I\nabla_t+\nabla_j\nabla_j)W_m
=\frac{\I}{2}\,\Im\hpair{W_m}{W_j}\,W_j,
\end{equation}
and
\begin{equation}\label{eq:cov-heat-W-en}
(\nabla_s-\nabla_j\nabla_j)W_m
=\frac12|W_j|^2W_m-\frac12\Re\hpair{W_m}{W_j}\,W_j.
\end{equation}
\end{proposition}

\begin{proof}
Apply the operator $I+\I J_u$ to \eqref{eq:cov-sch-psi-en} and use Lemma \ref{lem:Jparallel-en}. Since
\[
(I+\I J_u)J_u=-\I(I+\I J_u),
\]
the principal part becomes $-\I\nabla_j\nabla_jW_m$, and thus
\[
(\I\nabla_t+\nabla_j\nabla_j)W_m
=\I(I+\I J_u)R(\psi_m,\psi_j)J_u\psi_j.
\]
By the spherical curvature formula \eqref{eq:curvature-en},
\[
R(\psi_m,\psi_j)J_u\psi_j
=-\ip{\psi_m}{J_u\psi_j}\,\psi_j.
\]
Using \eqref{eq:WW-en} and $(I+\I J_u)\psi_j=W_j$, one obtains \eqref{eq:cov-sch-W-en}.

The heat-flow equation \eqref{eq:cov-heat-W-en} follows similarly from \eqref{eq:cov-heat-psi-en}, because
\[
R(\psi_m,\psi_j)\psi_j=|\psi_j|^2\psi_m-\ip{\psi_m}{\psi_j}\psi_j.
\qedhere
\]
\end{proof}

The covariant equations will be used through their ambient $\C^3$ form. Expanding the covariant system back into the ambient space rewrites the covariant Laplacian as the ordinary Laplacian plus connection terms, and it separates explicitly the matrix-valued first-order transport term that has to be closed. By \eqref{eq:intro-complex-notation-en}, for any $\C^3$-valued field $X$ one has
\[
\Gamma_\alpha X
= u\,\ip{\partial_\alpha u}{X}-(\partial_\alpha u)\,\ip{u}{X}.
\]
In particular, if $X$ is a tangential field along $u$, then
\[
\nabla_\alpha X = \partial_\alpha X+\Gamma_\alpha X.
\]
Moreover,
\begin{equation}\label{eq:Gammasize-en}
|\Gamma_m|\lesssim |\partial_mu|\sim |W_m|.
\end{equation}

\begin{lemma}[Connection algebra]\label{lem:gamma-basic-en}
If $Z$ is a complex tangential field along $u$, then
\begin{equation}\label{eq:gammaZ-body-en}
\Gamma_j Z=u\,\ip{\psi_j}{Z},
\qquad
\Gamma_j^2 Z=-\psi_j\,\ip{\psi_j}{Z},
\end{equation}
and therefore
\begin{equation}\label{eq:gammaZ-bound-en}
|\Gamma_jZ|\le |\psi_j|\,|Z|,
\qquad
|\Gamma_j^2Z|\le |\psi_j|^2|Z|.
\end{equation}
In addition,
\begin{equation}\label{eq:divGamma-body-en}
\partial_j\Gamma_j
=u\otimes(\Delta u)^\top-(\Delta u)\otimes u^\top.
\end{equation}
\end{lemma}

\begin{proof}
Since $Z$ is tangential, $\ip{u}{Z}=0$, and hence
\[
\Gamma_jZ
=u\,\ip{\psi_j}{Z}-\psi_j\,\ip{u}{Z}
=u\,\ip{\psi_j}{Z},
\]
which is the first identity in \eqref{eq:gammaZ-body-en}. Applying $\Gamma_j$ once more, and using $\ip{\psi_j}{u}=0$ and $\ip{u}{u}=1$, yields
\[
\Gamma_j^2Z
=\Gamma_j\bigl(u\,\ip{\psi_j}{Z}\bigr)
=-\psi_j\,\ip{\psi_j}{Z},
\]
which proves the second identity in \eqref{eq:gammaZ-body-en} and therefore \eqref{eq:gammaZ-bound-en}.

Differentiating the explicit formula for $\Gamma_j$ gives
\[
\partial_j\Gamma_j
=\psi_j\otimes\psi_j^\top+u\otimes(\partial_j\psi_j)^\top
-(\partial_j\psi_j)\otimes u^\top-\psi_j\otimes\psi_j^\top,
\]
and the quadratic terms cancel pairwise. This yields \eqref{eq:divGamma-body-en}.
\end{proof}

\begin{proposition}[Ambient equations]\label{prop:ambient-W-en}
After expanding the covariant equations from the previous subsection into the ambient space, the complex tangential fields $W_m$ satisfy
\begin{equation}\label{eq:ambient-W-en}
(\I\partial_t+\Delta)W_m
=-2\Gamma_j\partial_jW_m-\Gamma_j^2W_m
+\frac{\I}{2}\,\Im\hpair{W_m}{W_j}\,W_j.
\end{equation}
Along the harmonic map heat flow one also has
\begin{equation}\label{eq:ambient-heat-en}
(\partial_s-\Delta)W_m
=2\Gamma_j\partial_jW_m + V_m^{(s)},
\end{equation}
where
\begin{equation}\label{eq:Vsbound-en}
|V_m^{(s)}|
\lesssim \bigl(|\partial_j\Gamma_j|+|\Gamma_s|+|\Gamma_j|^2+|W_x|^2\bigr)|W_m|.
\end{equation}
\end{proposition}

\begin{proof}
Expand the covariant Laplacian in \eqref{eq:cov-sch-W-en}:
\[
\nabla_j\nabla_jW_m
=\Delta W_m+2\Gamma_j\partial_jW_m+(\partial_j\Gamma_j+\Gamma_j^2)W_m.
\]
Since $\nabla_tW_m=\partial_tW_m+\Gamma_tW_m$, it follows that
\[
(\I\partial_t+\Delta)W_m
=-2\Gamma_j\partial_jW_m-(\I\Gamma_t+\partial_j\Gamma_j+\Gamma_j^2)W_m
+\frac{\I}{2}\,\Im\hpair{W_m}{W_j}\,W_j.
\]
Thus it remains only to prove the cancellation
\begin{equation}\label{eq:cancel-en}
(\I\Gamma_t+\partial_j\Gamma_j)W_m=0.
\end{equation}
Since $W_m$ is tangential, Lemma \ref{lem:gamma-basic-en} and \eqref{eq:divGamma-body-en} give
\[
\Gamma_tW_m = u\,\ip{\partial_tu}{W_m},
\qquad
(\partial_j\Gamma_j)W_m = u\,\ip{\Delta u}{W_m}.
\]
Now $\partial_tu=J_u\Delta u$, and \eqref{eq:JW-en} implies that for any tangential vector $Y$,
\[
\ip{J_uY}{W_m} = -\ip{Y}{J_uW_m}=\I\ip{Y}{W_m}.
\]
Therefore
\[
\I\ip{\partial_tu}{W_m}+\ip{\Delta u}{W_m}
=\I\ip{J_u\Delta u}{W_m}+\ip{\Delta u}{W_m}
=\I(\I\ip{\Delta u}{W_m})+\ip{\Delta u}{W_m}=0,
\]
which proves \eqref{eq:cancel-en} and hence \eqref{eq:ambient-W-en}.

The heat-flow equation \eqref{eq:ambient-heat-en} is entirely analogous, except that the above cancellation is not needed there, so all zeroth-order terms are absorbed into $V_m^{(s)}$. The bound \eqref{eq:Vsbound-en} follows immediately from the expansion and \eqref{eq:Gammasize-en}.
\end{proof}

\begin{remark}
The significance of \eqref{eq:ambient-W-en} is that, after the cancellation \eqref{eq:cancel-en}, all truly nonperturbative terms that must be closed are concentrated in
\[
-2\Gamma_j\partial_jW_m.
\]
This is exactly parallel to the role played by $2A_j\partial_j\psi_m$ in the caloric gauge.
\end{remark}

We finally record the conservation law.

\begin{proposition}[Energy balance]\label{prop:energy-balance-en}
Define
\begin{equation}\label{eq:energy-density-en}
e[u]:=\frac12|\psi_x|^2=\frac14\sum_{m=1}^2\hpair{W_m}{W_m},
\qquad
q_j[u]:=-\ip{J_u\nabla_j\psi_m}{\psi_m}
=-\frac12\Im\hpair{\nabla_jW_m}{W_m}.
\end{equation}
Then one has the local conservation law
\begin{equation}\label{eq:energy-balance-en}
\partial_t e[u]+\partial_jq_j[u]=0.
\end{equation}
Equivalently, in terms of $W$,
\begin{equation}\label{eq:mass-balance-en}
\partial_t\frac14\sum_{m=1}^2\hpair{W_m}{W_m}
-
\frac12\partial_j\Im\sum_{m=1}^2\hpair{\nabla_jW_m}{W_m}=0.
\end{equation}
\end{proposition}

\begin{proof}
We first explain the second identity in \eqref{eq:energy-density-en}. For any tangential vectors $A,X$, write
$W_A:=A+\I J_uA$ and $W_X:=X+\I J_uX$. Then one has the algebraic identity
\[
-\ip{J_uA}{X}=-\frac12\Im\hpair{W_A}{W_X}.
\]
Substituting $A=\nabla_j\psi_m$ and $X=\psi_m$ gives
$q_j[u]= -\frac12\Im\hpair{\nabla_jW_m}{W_m}$.

Now start from Proposition \ref{prop:covdiff-en}:
\[
\partial_t e[u]
=\ip{\nabla_t\psi_m}{\psi_m}
=\ip{J_u\nabla_j\nabla_j\psi_m}{\psi_m}
+\ip{R(\psi_m,\psi_j)J_u\psi_j}{\psi_m}.
\]
Here both $m$ and $j$ are summed under the Einstein convention. The curvature term can be written as
\[
\ip{R(\psi_m,\psi_j)J_u\psi_j}{\psi_m}
=-\ip{\psi_m}{J_u\psi_j}\,\ip{\psi_j}{\psi_m}.
\]
Set
\[
a_{mj}:=\ip{\psi_m}{J_u\psi_j},\qquad s_{mj}:=\ip{\psi_j}{\psi_m}.
\]
Then $a_{mj}=-a_{jm}$ is an antisymmetric matrix, whereas $s_{mj}=s_{jm}$ is symmetric, and hence
\[
\sum_{m,j}a_{mj}s_{mj}=0.
\]
Therefore the total curvature contribution vanishes.

For the principal part,
\[
\ip{J_u\nabla_j\nabla_j\psi_m}{\psi_m}
=\partial_j\ip{J_u\nabla_j\psi_m}{\psi_m}
-\ip{J_u\nabla_j\psi_m}{\nabla_j\psi_m},
\]
and the last term vanishes because $\ip{J_uX}{X}=0$ for every tangential vector $X$. Thus
\[
\partial_t e[u]+\partial_jq_j[u]=0.
\]
Substituting the $W$-expressions for $e[u]$ and $q_j[u]$ from \eqref{eq:energy-density-en} yields \eqref{eq:mass-balance-en}.
\end{proof}

\begin{remark}
A single $m$-component does not in general satisfy an independent local mass conservation law; the genuine antisymmetric cancellation occurs only after summing over both $m$ and $j$ in the total energy flux.
\end{remark}

Later, in the heat-flow estimates and the terminal-window analysis, we shall again use both forms \eqref{eq:energy-balance-en} and \eqref{eq:mass-balance-en} from Proposition \ref{prop:energy-balance-en}.

\subsection{Spaces and linear estimates}

We first prove a multiplier lemma, based only on Bony decomposition, which compresses the action of variable-coefficient multiplication on
\[
\ell^2L_{t,x}^4(I\times\R^2)
\]
into $L^\infty$ and $\dot H_x^1$ control.

\begin{lemma}[Besov multiplier]\label{lem:multiplier-en}
Let $a:I\times\R^2\to \C^{m\times m}$ satisfy
\[
a\in L^\infty_{t,x},
\qquad
\nabla_x a\in L_t^\infty L_x^2.
\]
Then for any $\C^m$-valued field $g$, one has
\begin{equation}\label{eq:multiplier-en}
\norm{ag}_{\ell^2L_{t,x}^4(I\times\R^2)}
\lesssim
\Bigl(
\norm{a}_{L_{t,x}^\infty}
+\norm{\nabla_x a}_{L_t^\infty L_x^2}
\Bigr)
\norm{g}_{\ell^2L_{t,x}^4(I\times\R^2)}.
\end{equation}
The implicit constant is independent of the interval $I$, the frequency scale, and the matrix dimension.
\end{lemma}

\begin{proof}
Write $\Delta_j:=P_j$, $S_{j-4}:=P_{<j-4}$, and use Bony's decomposition
\[
ag=T_ag+T_ga+R(a,g),
\]
where
\[
T_ag:=\sum_j S_{j-4}a\,\Delta_j g,
\qquad
T_ga:=\sum_j S_{j-4}g\,\Delta_j a,
\qquad
R(a,g):=\sum_{|j-\ell|\le 3}\Delta_j a\,\Delta_\ell g.
\]
See \cite{Bony,BCD}.

First consider the low-high interaction $T_ag$. By frequency support, $P_kT_ag$ receives contributions only from terms with $|j-k|\le 4$, and therefore
\[
\norm{P_kT_ag}_{L_{t,x}^4}
\lesssim
\sum_{|j-k|\le 4}
\norm{S_{j-4}a\,\Delta_j g}_{L_{t,x}^4}
\le
\norm{a}_{L_{t,x}^\infty}
\sum_{|j-k|\le 4}\norm{\Delta_j g}_{L_{t,x}^4}.
\]
Taking the $\ell^2$ norm in $k$ and using finite overlap gives
\[
\norm{T_ag}_{\ell^2L_{t,x}^4}
\lesssim
\norm{a}_{L_{t,x}^\infty}\norm{g}_{\ell^2L_{t,x}^4}.
\]

Next consider the high-low interaction $T_ga$. Again by frequency support,
\[
P_kT_ga
=
\sum_{|j-k|\le 4}P_k\bigl(S_{j-4}g\,\Delta_j a\bigr).
\]
Hence
\[
\norm{P_kT_ga}_{L_{t,x}^4}
\lesssim
\sum_{|j-k|\le 4}
\norm{S_{j-4}g}_{L_{t,x}^4}\norm{\Delta_j a}_{L_t^\infty L_x^\infty}.
\]
By the Littlewood--Paley square-function estimate,
\[
\norm{g}_{L_{t,x}^4}
\lesssim
\Biggl\|
\Bigl(\sum_\ell |\Delta_\ell g|^2\Bigr)^{1/2}
\Biggr\|_{L_{t,x}^4}
\lesssim
\Bigl(\sum_\ell \norm{\Delta_\ell g}_{L_{t,x}^4}^2\Bigr)^{1/2}
=
\norm{g}_{\ell^2L_{t,x}^4},
\]
and so $\norm{S_{j-4}g}_{L_{t,x}^4}\le \norm{g}_{L_{t,x}^4}\lesssim \norm{g}_{\ell^2L_{t,x}^4}$. By Bernstein,
\[
\norm{\Delta_j a}_{L_t^\infty L_x^\infty}
\lesssim
2^j\norm{\Delta_j a}_{L_t^\infty L_x^2}.
\]
Therefore
\[
\norm{P_kT_ga}_{L_{t,x}^4}
\lesssim
\norm{g}_{\ell^2L_{t,x}^4}
\sum_{|j-k|\le 4}2^j\norm{\Delta_j a}_{L_t^\infty L_x^2}.
\]
Taking the $\ell^2$ norm in $k$ and again using finite overlap yields
\[
\norm{T_ga}_{\ell^2L_{t,x}^4}
\lesssim
\norm{\nabla_x a}_{L_t^\infty L_x^2}\,
\norm{g}_{\ell^2L_{t,x}^4}.
\]

Finally consider the resonant term $R(a,g)$. Let
\[
\widetilde\Delta_j:=\Delta_{j-1}+\Delta_j+\Delta_{j+1}.
\]
Then
\[
R(a,g)=\sum_j \Delta_j a\,\widetilde\Delta_j g,
\]
and $P_kR(a,g)$ receives contributions only from terms with $|j-k|\le 5$. Thus
\[
\norm{P_kR(a,g)}_{L_{t,x}^4}
\lesssim
\sum_{|j-k|\le 5}
\norm{\Delta_j a}_{L_t^\infty L_x^\infty}
\norm{\widetilde\Delta_j g}_{L_{t,x}^4}
\lesssim
\sum_{|j-k|\le 5}
2^j\norm{\Delta_j a}_{L_t^\infty L_x^2}
\norm{\widetilde\Delta_j g}_{L_{t,x}^4}.
\]
Taking the $\ell^2$ norm in $k$ and using finite overlap once more gives
\[
\norm{R(a,g)}_{\ell^2L_{t,x}^4}
\lesssim
\norm{\nabla_x a}_{L_t^\infty L_x^2}\,
\norm{g}_{\ell^2L_{t,x}^4}.
\]
Adding the three parts yields \eqref{eq:multiplier-en}.
\end{proof}

The next lemma pulls the whole argument back to the original quantity $\nabla u$.

\begin{lemma}[Norm comparison]\label{lem:compare-en}
Fix $E_0<4\pi$. Then there exists a constant $C(E_0)\ge 1$ such that, if $E(u)\le E_0$ and $u-Q\in L_x^2$, then for every compact interval $I$ one has
\begin{equation}\label{eq:compare-en}
C(E_0)^{-1}\norm{\nabla u}_{\ell^2L^4_{t,x}(I\times\R^2)}
\le
\norm{W_x}_{\ell^2L^4_{t,x}(I\times\R^2)}
\le
C(E_0)\norm{\nabla u}_{\ell^2L^4_{t,x}(I\times\R^2)}.
\end{equation}
\end{lemma}

\begin{proof}
By \eqref{eq:intro-basic-notation-en}--\eqref{eq:intro-complex-notation-en}, and writing $J_QX:=Q\times X$, one has
\[
W_x=(I+\I J_Q)\nabla u + \I\bigl((u-Q)\times \nabla u\bigr).
\]
Here $I+\I J_Q$ is a constant-coefficient zeroth-order multiplier, and hence
\[
\norm{(I+\I J_Q)\nabla u}_{\ell^2L_{t,x}^4}
\lesssim
\norm{\nabla u}_{\ell^2L_{t,x}^4}.
\]
On the other hand, pointwise one has $|u-Q|\le 2$, while
\[
\norm{\nabla_x(u-Q)}_{L_t^\infty L_x^2}
=
\norm{\nabla u}_{L_t^\infty L_x^2}
\le (2E_0)^{1/2}.
\]
Applying Lemma \ref{lem:multiplier-en} with
\[
a=u-Q,
\qquad
g=\nabla u,
\]
one obtains
\[
\norm{(u-Q)\times \nabla u}_{\ell^2L_{t,x}^4}
\lesssim
\bigl(1+E_0^{1/2}\bigr)\norm{\nabla u}_{\ell^2L_{t,x}^4}.
\]
Thus
\[
\norm{W_x}_{\ell^2L_{t,x}^4}
\lesssim_{E_0}
\norm{\nabla u}_{\ell^2L_{t,x}^4}.
\]

The reverse estimate is even more direct. Since
\[
\partial_m u=\frac12\bigl(W_m+\overline{W_m}\bigr),
\]
and the Littlewood--Paley projections are real Fourier multipliers, one has
\[
P_k\partial_m u=\frac12\bigl(P_kW_m+\overline{P_kW_m}\bigr).
\]
Hence for each $k$,
\[
\norm{P_k\partial_m u}_{L_{t,x}^4}
\le
\frac12\norm{P_kW_m}_{L_{t,x}^4}
+\frac12\norm{\overline{P_kW_m}}_{L_{t,x}^4}
=
\norm{P_kW_m}_{L_{t,x}^4}.
\]
Summing over $m=1,2$ and then taking the $\ell^2$ norm in $k$ yields
\[
\norm{\nabla u}_{\ell^2L_{t,x}^4}
\lesssim
\norm{W_x}_{\ell^2L_{t,x}^4}.
\]
Combining the two bounds proves \eqref{eq:compare-en}.
\end{proof}

We now turn to the analytic framework. The dyadic resolution space used below is assembled from Strichartz estimates, lateral Strichartz estimates, local smoothing, and Galilean maximal-function bounds. All norms first act componentwise on $\C^3$-valued fields, and then the Euclidean sum over the three ambient components is taken.

For $\theta\in\mathbb S^1$, let $H_\theta$ denote the orthogonal complement of $\theta$. For $p,q\in[1,\infty]$, define
\[
\norm{f}_{L_\theta^{p,q}}
:=
\Biggl(
\int_{\R}
\Bigl(
\int_{H_\theta\times\R}|f(t,x_1\theta+x_2)|^q\,dx_2\,dt
\Bigr)^{p/q}\,dx_1
\Biggr)^{1/p}.
\]
For $\lambda\in\R$, define the Galilean transform
\[
T_{\lambda\theta}f(t,x)
:= e^{-\frac{\I}{2}x\cdot(\lambda\theta)}e^{-\frac{\I}{4}t|\lambda|^2}
   f(t,x+t\lambda\theta),
\]
and the corresponding norm
\[
\norm{f}_{L_{\theta,\lambda}^{p,q}}:=\norm{T_{\lambda\theta}f}_{L_\theta^{p,q}}.
\]
Fix a large integer $K\gg1$, and let
\[
W_k:=\{\lambda\in[-2^k,2^k]:2^{k+2K}\lambda\in\mathbb Z\}.
\]
If $\theta\in\mathbb S^1$, then $P_{j,\theta}$ denotes a fixed smooth sectorial projection with finite overlap, whose Fourier support satisfies $|\xi|\sim2^j$ and $|\xi\cdot\theta|\sim2^j$. All constants below are independent of this finite-overlap decomposition.

If $I\subset[0,\infty)$ and
\[
\Xi_k:=\{\xi\in\R^2:|\xi|\in[2^{k-1},2^{k+1}]\},
\]
set
\[
L_k^2(I):=\{f\in L^2(I\times\R^2):\operatorname{supp}\widehat f(t,\xi)\subset I\times\Xi_k\}.
\]
For $f\in L_k^2(I)$ define
\begin{align*}
\norm{f}_{X_k^0(I)}
&:= \norm{f}_{L_t^\infty L_x^2}
 + \norm{f}_{L^4_{t,x}}
 + 2^{-k/2}\norm{f}_{L_x^4L_t^\infty}
 + 2^{-k/2}\sup_{\theta\in\mathbb S^1}\norm{f}_{L_{\theta,W_{k+40}}^{2,\infty}},
\\
\norm{f}_{X_k(I)}
&:= \norm{f}_{X_k^0(I)}
 + 2^{-k/6}\sup_{\theta\in\mathbb S^1}\norm{f}_{L_\theta^{3,6}}
 + 2^{k/6}\sup_{|j-k|\le 20}\sup_{\theta\in\mathbb S^1}\norm{P_{j,\theta}f}_{L_\theta^{6,3}}
\\
&\qquad
 + 2^{k/2}\sup_{|j-k|\le20}\sup_{\theta\in\mathbb S^1}\sup_{|\lambda|<2^{k-40}}
   \norm{P_{j,\theta}f}_{L_{\theta,\lambda}^{\infty,2}}.
\end{align*}
The space $Y_k(I)$ is defined by atomic decomposition:
\begin{align*}
\norm{f}_{Y_k(I)}
:= \inf_{f=f_1+f_2+f_3+f_4}
\Bigl(
&\norm{f_1}_{L_{t,x}^{4/3}}
+2^{k/6}\norm{f_2}_{L_{\hat\theta_1}^{3/2,6/5}}
+2^{k/6}\norm{f_3}_{L_{\hat\theta_2}^{3/2,6/5}}
\\
&+2^{-k/2}\sup_{\theta\in\mathbb S^1}\norm{f_4}_{L_{\theta,W_{k-40}}^{1,2}}
\Bigr),
\end{align*}
where $(\hat\theta_1,\hat\theta_2)$ is the standard basis of $\R^2$.

We now prove the linear estimates needed for free propagation: the standard $L^4_{t,x}$ Strichartz estimate, lateral Strichartz estimates, local smoothing, and the Galilean maximal-function endpoint. These estimates go back to the classical Strichartz and local-smoothing theory and its refinements; see, for instance, \cite{Strichartz,KPV,ChristKiselev,PlanchonVega}. The proof below uses only Fourier representation, $TT^*$, the coarea formula, interpolation, and cross-section estimates.

\begin{lemma}[Free propagation]\label{lem:free-linear-en}
Let
\[
S_k(t)f:=1_I(t)e^{\I(t-t_0)\Delta}P_k f.
\]
Then for every $f\in L_x^2(\R^2;\C^3)$,
\begin{align}
\norm{S_k f}_{L_t^\infty L_x^2}
+\norm{S_k f}_{L_{t,x}^4}
&\lesssim \norm{f}_{L_x^2},\label{eq:freeL4-en}\\
2^{-k/6}\sup_{\theta\in\mathbb S^1}\norm{S_k f}_{L_\theta^{3,6}}
+
2^{k/6}\sup_{|j-k|\le 20}\sup_{\theta\in\mathbb S^1}\norm{P_{j,\theta}S_kf}_{L_\theta^{6,3}}
&\lesssim \norm{f}_{L_x^2},\label{eq:freelateral-en}\\
2^{k/2}\sup_{|j-k|\le20}\sup_{\theta\in\mathbb S^1}\sup_{|\lambda|<2^{k-40}}
\norm{P_{j,\theta}S_kf}_{L_{\theta,\lambda}^{\infty,2}}
+
2^{-k/2}\sup_{\theta\in\mathbb S^1}\norm{S_kf}_{L_{\theta,W_{k+40}}^{2,\infty}}
&\lesssim \norm{f}_{L_x^2}.\label{eq:freesmooth-en}
\end{align}
In particular, by Bernstein,
\begin{equation}\label{eq:free-max-en}
2^{-k/2}\norm{S_kf}_{L_x^4L_t^\infty}
\lesssim
\norm{S_kf}_{L_t^\infty L_x^2}
\lesssim
\norm{f}_{L_x^2}.
\end{equation}
\end{lemma}

\begin{proof}
The proof is divided into four steps. Since time translation, rotation, and scaling preserve the form of these mixed norms, it is enough to prove the scalar case with $t_0=0$ and $k=0$; the vector-valued case follows by summing over components.

First, unitarity directly gives
\[
\norm{S_0 f}_{L_t^\infty L_x^2}\le \norm{f}_{L_x^2}.
\]
Let $Tf(t,x):=1_I(t)e^{\I t\Delta}P_0f(x)$. By the $TT^*$ principle, it suffices to prove
\[
\norm{TT^*F}_{L_{t,x}^4}\lesssim \norm{F}_{L_{t,x}^{4/3}}.
\]
Since
\[
(TT^*F)(t)=\int_I e^{\I(t-s)\Delta}P_0^2F(s)\,ds,
\]
and the kernel of $e^{\I\tau\Delta}P_0^2$ satisfies the stationary phase bound
\[
|K(\tau,x)|\lesssim_N (1+|\tau|)^{-1}\bigl(1+|x|/(1+|\tau|)\bigr)^{-N},
\]
we have
\[
\norm{e^{\I\tau\Delta}P_0^2 g}_{L_x^4}
\lesssim (1+|\tau|)^{-1/2}\norm{g}_{L_x^{4/3}}.
\]
Thus
\[
\norm{TT^*F(t)}_{L_x^4}
\lesssim \int_I (1+|t-s|)^{-1/2}\norm{F(s)}_{L_x^{4/3}}\,ds.
\]
The one-dimensional Hardy--Littlewood--Sobolev inequality in time gives
\[
\norm{TT^*F}_{L_t^4L_x^4}
\lesssim \norm{F}_{L_t^{4/3}L_x^{4/3}},
\]
which proves \eqref{eq:freeL4-en}.

Second, we prove the local-smoothing endpoint. Fix $|j|\le20$ and $\theta=e_1$, and set
\[
\nu(t,x):=P_{j,e_1}e^{\I t\Delta}P_0f.
\]
Then $\widehat\nu$ is supported where $|\xi|\sim1$ and $|\xi_1|\sim1$. For fixed $x_1$, taking the Fourier transform in $(x_2,t)$ gives
\[
\mathcal F_{x_2,t}\nu(x_1,\eta,\tau)
=\int e^{\I x_1\xi_1}m_j(\xi_1,\eta)\widehat f(\xi_1,\eta)
\,\delta(\tau-\xi_1^2-\eta^2)\,d\xi_1.
\]
Since $|\xi_1|\sim1$, the coarea formula and Cauchy--Schwarz imply
\[
\int_{\R^2}|\mathcal F_{x_2,t}\nu(x_1,\eta,\tau)|^2\,d\eta\,d\tau
\lesssim \norm{f}_{L_x^2}^2.
\]
By Plancherel in $(x_2,t)$,
\[
\norm{P_{j,e_1}e^{\I t\Delta}P_0f}_{L_{e_1}^{\infty,2}}
\lesssim \norm{f}_{L_x^2}.
\]
Scaling gives
\[
\norm{P_{j,\theta}S_kf}_{L_{\theta,0}^{\infty,2}}
\lesssim 2^{-k/2}\norm{f}_{L_x^2},
\qquad |j-k|\le20.
\]
The Galilean commutation relation with the free Schrödinger group then yields
\[
\norm{P_{j,\theta}S_kf}_{L_{\theta,\lambda}^{\infty,2}}
\lesssim 2^{-k/2}\norm{f}_{L_x^2},
\qquad |j-k|\le20,\\ |\lambda|<2^{k-40}.
\]
This is the local-smoothing part of \eqref{eq:freesmooth-en}.

Third, we prove the lateral Strichartz estimates and the two-dimensional Galilean maximal-function endpoint. Again take $k=0$ and $\theta=e_1$. For the Galilean endpoint, let $I\subset[0,2^{2K}]$ and set $\delta:=2^{-2K-40}$. Decompose the unit annulus in the variable $2\xi_1$ as
\[
1=\sum_{\lambda\in W_{40}}\chi_\lambda(2\xi_1),
\qquad
\operatorname{supp}\chi_\lambda\subset\{|2\xi_1-\lambda|\le4\delta\},
\]
with uniformly bounded overlap. Let $f_\lambda$ be defined by
\[
\widehat f_\lambda(\xi)=\chi_\lambda(2\xi_1)\widehat{P_0f}(\xi),
\qquad
\nu_\lambda(t,x):=1_I(t)e^{\I t\Delta}f_\lambda(x).
\]
It suffices to prove
\begin{equation}\label{eq:gal-piece-en}
\norm{\nu_\lambda}_{L_{e_1,\lambda}^{2,\infty}}
\lesssim \delta^{1/2}\norm{f_\lambda}_{L_x^2}.
\end{equation}
Indeed, summing over $\lambda$ and using Cauchy--Schwarz, the effective number $O(\delta^{-1})$ of strips, and almost orthogonality yields
\[
\norm{S_0f}_{L_{e_1,W_{40}}^{2,\infty}}
\lesssim \norm{f}_{L_x^2}.
\]
Scaling gives
\begin{equation}\label{eq:gal-max-proof-en}
2^{-k/2}\sup_{\theta\in\mathbb S^1}
\norm{S_kf}_{L_{\theta,W_{k+40}}^{2,\infty}}
\lesssim \norm{f}_{L_x^2}.
\end{equation}

To prove \eqref{eq:gal-piece-en}, apply the Galilean transform to $\nu_\lambda$ and ignore unimodular phases. One obtains
\[
T_{\lambda e_1}\nu_\lambda(t,x)
=1_I(t)\int e^{\I(x_1a/2+x_2\eta+t(\eta^2+a^2/4))}
 b_\lambda(a,\eta)\,da\,d\eta,
\]
where $|a|\lesssim\delta$, $|\eta|\lesssim1$, and $\norm{b_\lambda}_{L^2_{a,\eta}}\sim\norm{f_\lambda}_{L_x^2}$. For fixed $x_1$, the map
\[
(a,\eta)\mapsto(\tau,\eta):=(\eta^2+a^2/4,\eta)
\]
has fiber length $O(\delta)$ when $|a|\lesssim\delta$. The coarea formula and Cauchy--Schwarz therefore give
\[
\norm{T_{\lambda e_1}\nu_\lambda(x_1,\cdot)}_{H^2_{t,x_2}}
\lesssim
\delta^{1/2}
\Bigl(\int |\mathcal F^{-1}_{a\to x_1}b_\lambda(x_1,\eta)|^2\,d\eta\Bigr)^{1/2}.
\]
The embedding $H^2_{t,x_2}\hookrightarrow L^\infty_{t,x_2}$, followed by $L^2$ in $x_1$ and Plancherel, proves \eqref{eq:gal-piece-en}.

For the lateral Strichartz estimates, first assume that the Fourier support lies in a fixed angular sector with $|\xi_1|\sim1$. Let
\[
Ef(t,x_1,x_2):=e^{\I t\Delta}P_0f(x_1,x_2).
\]
For fixed $x_1$,
\[
\mathcal F_{x_2,t}Ef(x_1,\eta,\tau)
=\int e^{\I x_1\xi_1}m(\xi_1,\eta)\widehat f(\xi_1,\eta)
\delta(\tau-\xi_1^2-\eta^2)\,d\xi_1.
\]
Since $|\xi_1|\sim1$, the change of variables $(\xi_1,\eta)\mapsto(\tau,\eta)$ has Jacobian $2|\xi_1|\sim1$. This nondegeneracy, together with Hausdorff--Young, the coarea formula, and $TT^*$, gives
\[
\norm{Ef}_{L_{e_1}^{p,q}}
\lesssim
\norm{f}_{L_x^2},
\qquad
2<p\le\infty,
\quad 2\le q<\infty,
\quad \frac1p+\frac1q=\frac12.
\]
After a finite angular decomposition and rotation,
\[
\norm{S_0 f}_{L_\theta^{3,6}}
+
\sup_{|j|\le20}\norm{P_{j,\theta}S_0f}_{L_\theta^{6,3}}
\lesssim \norm{f}_{L_x^2}.
\]
Scaling gives
\begin{equation}\label{eq:lateral-proof-en}
2^{-k/6}\sup_{\theta\in\mathbb S^1}\norm{S_k f}_{L_\theta^{3,6}}
+
2^{k/6}\sup_{|j-k|\le20}\sup_{\theta\in\mathbb S^1}\norm{P_{j,\theta}S_kf}_{L_\theta^{6,3}}
\lesssim \norm{f}_{L_x^2}.
\end{equation}
Equations \eqref{eq:gal-max-proof-en} and \eqref{eq:lateral-proof-en} imply \eqref{eq:freelateral-en} and the remaining part of \eqref{eq:freesmooth-en}.

Finally, Bernstein gives
\[
\norm{P_k h}_{L_x^4}\lesssim 2^{k/2}\norm{P_k h}_{L_x^2}.
\]
Taking $h=e^{\I(t-t_0)\Delta}P_kf$ and then the supremum in $t$ proves \eqref{eq:free-max-en}.
\end{proof}

\begin{proposition}[Linear estimate]\label{prop:linear-en}
Let $I=[t_0,t_1]\subset (0,2^{2K}]$, let $k\in\mathbb Z$, and suppose $U_0\in L_x^2(\R^2;\C^3)$ is spatially frequency-supported in $\Xi_k$. If $V$ satisfies
\[
(\I\partial_t+\Delta)V=H,
\qquad
V(t_0)=U_0,
\]
then
\begin{equation}\label{eq:linear-en}
\norm{V}_{X_k(I)}\lesssim \norm{U_0}_{L_x^2}+\norm{H}_{Y_k(I)}.
\end{equation}
\end{proposition}

\begin{proof}
This follows from Lemma \ref{lem:free-linear-en} by the standard cross $TT^*$ argument and the atomic definition of $Y_k$. In more detail, write Duhamel's formula
\[
V(t)=e^{\I(t-t_0)\Delta}U_0-
\I\int_{t_0}^t e^{\I(t-s)\Delta}H(s)\,ds.
\]
The homogeneous part is bounded in $X_k(I)$ by Lemma \ref{lem:free-linear-en}. For the inhomogeneous part, decompose $H=H_1+H_2+H_3+H_4$ according to the four atoms in the definition of $Y_k$. Applying the free estimates with $TT^*$ gives the contributions
\[
\norm{H_1}_{L_{t,x}^{4/3}},
\qquad
2^{k/6}\norm{H_2}_{L_{\hat\theta_1}^{3/2,6/5}},
\qquad
2^{k/6}\norm{H_3}_{L_{\hat\theta_2}^{3/2,6/5}},
\]
and
\[
2^{-k/2}\sup_{\theta\in\mathbb S^1}\norm{H_4}_{L_{\theta,W_{k-40}}^{1,2}}.
\]
Taking the infimum over all such decompositions yields \eqref{eq:linear-en}.
\end{proof}

\begin{lemma}[Product estimate]\label{lem:product-en}
For $h\in L^2_{t,x}$ and $f\in X_j(I)$, one has
\begin{equation}\label{eq:product-en}
\norm{P_k(hf)}_{Y_k(I)}
\lesssim
\begin{cases}
\norm{h}_{L^2_{t,x}}\norm{f}_{X_j(I)}, & |j-k|\le 80,\\[0.3em]
2^{-\frac12|j-k|}\norm{h}_{L^2_{t,x}}\norm{f}_{X_j(I)}, & j\le k-80,\\[0.3em]
2^{-\frac16|j-k|}\norm{h}_{L^2_{t,x}}\norm{f}_{X_j(I)}, & k\le j-80.
\end{cases}
\end{equation}
\end{lemma}

The estimate \eqref{eq:product-en} is used repeatedly below to place products containing an $L^2_{t,x}$ factor into the forcing space $Y_k$.

\begin{proof}
If $|j-k|\le 80$, directly use the first atom of $Y_k$:
\[
\norm{P_k(hf)}_{Y_k}
\le
\norm{P_k(hf)}_{L_{t,x}^{4/3}}
\le
\norm{h}_{L_{t,x}^2}\norm{f}_{L_{t,x}^4}
\le
\norm{h}_{L_{t,x}^2}\norm{f}_{X_j}.
\]

If $j\le k-80$, use the fourth atom. For arbitrary $\theta$ and $\lambda$, mixed-norm H\"older gives
\[
\norm{P_k(hf)}_{L_{\theta,\lambda}^{1,2}}
\le
\norm{h}_{L_{t,x}^2}\norm{f}_{L_{\theta,\lambda}^{2,\infty}}.
\]
Taking the $W_{k-40}$ sum-space norm and using $W_{j+40}\subset W_{k-40}$,
\begin{align*}
\norm{P_k(hf)}_{Y_k}
&\le
2^{-k/2}\sup_{\theta}\norm{P_k(hf)}_{L_{\theta,W_{k-40}}^{1,2}}\\
&\le
2^{-(k-j)/2}\norm{h}_{L_{t,x}^2}
\Bigl(2^{-j/2}\sup_{\theta}\norm{f}_{L_{\theta,W_{j+40}}^{2,\infty}}\Bigr)\\
&\le
2^{-\frac12(k-j)}\norm{h}_{L_{t,x}^2}\norm{f}_{X_j}.
\end{align*}

If $k\le j-80$, use the second atom and mixed-norm H\"older:
\[
\norm{P_k(hf)}_{L_{\hat\theta_1}^{3/2,6/5}}
\le
\norm{h}_{L_{t,x}^2}\norm{f}_{L_{\hat\theta_1}^{3,6}}.
\]
Therefore
\begin{align*}
\norm{P_k(hf)}_{Y_k}
&\le
2^{k/6}\norm{P_k(hf)}_{L_{\hat\theta_1}^{3/2,6/5}}\\
&\le
2^{-(j-k)/6}\norm{h}_{L_{t,x}^2}
\Bigl(2^{-j/6}\sup_{\theta}\norm{f}_{L_\theta^{3,6}}\Bigr)\\
&\le
2^{-\frac16(j-k)}\norm{h}_{L_{t,x}^2}\norm{f}_{X_j}.
\end{align*}
This proves all three cases.
\end{proof}

For $\sigma\ge0$ and a compact interval $I=[t_0,t_1]$, define the direct frequency envelopes
\begin{align}
b_k(\sigma)
&:= \sup_{j\in\mathbb Z}2^{\sigma j}2^{-\delta|k-j|}\norm{P_jW_x}_{X_j(I)},\label{eq:bk-en}\\
\alpha_k(\sigma)
&:= \sup_{j\in\mathbb Z}2^{\sigma j}2^{-\delta|k-j|}\norm{P_jW_x}_{L_{t,x}^4(I\times\R^2)},\label{eq:ak-en}\\
\beta_k(\sigma)
&:= \sup_{j\in\mathbb Z}2^{\sigma j}2^{-\delta|k-j|}\norm{P_jW_x(t_0)}_{L_x^2}.\label{eq:betak-en}
\end{align}
We write
\[
b_k:=b_k(0),
\quad
\alpha_k:=\alpha_k(0),
\quad
\beta_k:=\beta_k(0),
\quad
\upsilon_k(\sigma):=\alpha_k(\sigma)+\beta_k(\sigma).
\]
Then
\[
\sum_k\alpha_k^2\sim \norm{W_x}_{\ell^2L_{t,x}^4(I\times\R^2)}^2,
\qquad
\sum_k\beta_k^2\sim \norm{W_x(t_0)}_{L_x^2}^2\sim E(u(t_0)).
\]

\subsection{Envelopes and continuation}

We now record the frequency-envelope estimate to be closed later. This proposition is the common input for Theorems \ref{thm:globalize-en} and \ref{thm:window-global-en}.

\begin{proposition}[Frequency envelope]\label{prop:main-freq-en}
Fix $E_0<4\pi$ and $\sigma_1\in\mathbb Z_{\ge0}$. There exist $\delta,\eps_0>0$, depending only on $E_0$ and $\sigma_1$, such that the following holds. If
\[
u:I\times\R^2\to\Sph^2
\]
is a smooth solution of \eqref{eq:SM-en}, satisfying
\[
u(t_0)-Q\in L_x^2,
\qquad
u(t_0)\in \dotH^1\cap\dotH^3,
\qquad
E(u)\le E_0,
\]
and
\begin{equation}\label{eq:small-W-en}
\norm{W_x}_{\ell^2L_{t,x}^4(I\times\R^2)}\le \eps_0,
\end{equation}
then for all $0\le\sigma\le\sigma_1$,
\begin{equation}\label{eq:techconcl-en}
b_k(\sigma)\lesssim \upsilon_k(\sigma).
\end{equation}
\end{proposition}

By Lemma \ref{lem:compare-en}, condition \eqref{eq:small-W-en} is equivalent to the smallness condition in the original variable,
\[
\norm{\nabla u}_{\ell^2L_{t,x}^4(I\times\R^2)}\le c(E_0)\eps_0.
\]
Thus Proposition \ref{prop:main-freq-en} can also be stated purely in terms of $\nabla u$. We keep $W_x$ only to make the covariant structure in the later argument more direct. Once this proposition has been proved later, it immediately yields the following local continuation statement; Theorems \ref{thm:globalize-en} and \ref{thm:window-global-en} are derived through this proposition.

\begin{proposition}[Local continuation]\label{prop:cont-step-en}
Fix $E_0<4\pi$. Then there exist $\eps_0=\eps_0(E_0)>0$ and $T=T(E_0)>0$ such that the following holds. Let $I=[t_0,t_1]\subset[0,\infty)$, and let
\[
u:I\times\R^2\to\Sph^2
\]
be a smooth solution of \eqref{eq:SM-en} satisfying
\[
u(t_0)-Q\in L_x^2,
\qquad
u(t_0)\in \dotH^1\cap\dotH^3,
\qquad
E(u)\le E_0.
\]
If
\[
\norm{\nabla u}_{\ell^2L^4_{t,x}(I\times\R^2)}\le \eps_0,
\]
then $u$ can be uniquely extended to
\[
u\in L_t^\infty\bigl([t_0,t_1+T]:\dotH^1\cap\dotH^3\bigr).
\]
\end{proposition}

The proof of Proposition \ref{prop:cont-step-en} only combines Proposition \ref{prop:main-freq-en} with the classical local theory. To avoid a circular argument, we postpone this short proof until after Proposition \ref{prop:main-freq-en} has been proved.

\section{Heat flow and two-body estimates}

Decay in the heat-flow direction is the central intermediate layer in the closing argument. For each fixed $t$, let $u(s,t,x)$ be the harmonic map heat flow with initial data $u(t,\cdot)$:
\[
\partial_s u=\tau(u),
\qquad
u(0,t,x)=u(t,x).
\]
Below the threshold $4\pi$, this heat flow exists globally and converges to the constant state $Q$.

The next subsection isolates the parabolic part of the argument. It first records the heat-flow smoothing mechanism and then turns it into dyadic decay estimates for both $W_x$ and $\Gamma_x$.

\subsection{Heat-flow decay}

By Proposition \ref{prop:cov-W-en} and Proposition \ref{prop:ambient-W-en}, the fields $W_m$ satisfy the direct geometric equation \eqref{eq:ambient-heat-en} along the heat flow. In order to put the subsequent dyadic calculations in a form that can be used directly in the text, we first record a standard parabolic smoothing input.

\begin{lemma}[Heat smoothing]\label{lem:heat-smooth-en}
Under the small controlling-norm assumption \eqref{eq:small-W-en}, for every integer $r\ge0$ one has
\begin{align}
\sup_{s>0}s^{r/2}\norm{\nabla_x^rW_x(s)}_{L_t^\infty L_x^2(I\times\R^2)}
+\sup_{s>0}s^{(r+1)/2}\norm{\nabla_x^rW_x(s)}_{L_t^\infty L_x^\infty(I\times\R^2)}
&\lesssim_{E_0,r}1,\label{eq:heat-smooth-W-en}\\
\sup_{s>0}s^{r/2}\norm{\nabla_x^r\Gamma_x(s)}_{L_t^\infty L_x^2(I\times\R^2)}
+\sup_{s>0}s^{(r+1)/2}\norm{\nabla_x^r\Gamma_s(s)}_{L_t^\infty L_x^2(I\times\R^2)}
&\lesssim_{E_0,r}1,\label{eq:heat-smooth-Gamma-en}\\
\sup_{s>0}s^{(r+1)/2}\norm{\nabla_x^r\Gamma_x(s)}_{L_t^\infty L_x^\infty(I\times\R^2)}
&\lesssim_{E_0,r}1.\notag
\end{align}
\end{lemma}

\begin{proof}
Fix a physical time $t$ and regard $u(s,x):=u(s,t,x)$ as a harmonic map heat flow depending only on the heat variable $s$. The energy dissipation identity
\[
\frac{d}{ds}E(u(s))=-\int_{\R^2}|\partial_su(s,x)|^2\,dx
\]
implies
\begin{equation}\label{eq:heat-energy-monotone-proof-en}
\sup_{s>0}\norm{W_x(s)}_{L_x^2}\lesssim E_0^{1/2},
\qquad
\int_0^\infty\norm{\partial_su(s)}_{L_x^2}^2\,ds\lesssim E_0.
\end{equation}

The first step is the scale-invariant first-order bound. We claim that
\begin{equation}\label{eq:heat-gradient-critical-bound-en}
\sup_{s>0}s^{1/2}\norm{\nabla_xu(s)}_{L_x^\infty}\le C(E_0).
\end{equation}
If this failed, there would be points $(s_n,x_n)$ such that $s_n|\nabla_xu(s_n,x_n)|^2\to\infty$. Set
\[
\lambda_n:=|\nabla_xu(s_n,x_n)|^{-1},
\qquad
\nu_n(\sigma,y):=u(s_n+\lambda_n^2\sigma,x_n+\lambda_n y).
\]
Since $s_n/\lambda_n^2\to\infty$, each $\nu_n$ is defined at least on a fixed backward unit parabolic cylinder, and satisfies
\[
|\nabla_y\nu_n(0,0)|=1,
\qquad
E(\nu_n(\sigma))=E(u(s_n+\lambda_n^2\sigma))\le E_0<4\pi.
\]
By local parabolic compactness, a subsequence converges to a nontrivial ancient limit $\nu_\infty$. This limit is stationary in the heat variable, and hence gives a nonconstant finite-energy harmonic map $\R^2\to\Sph^2$ with energy not exceeding $E_0<4\pi$. This contradicts the fact that the least energy of a nonconstant sphere-valued harmonic map is $4\pi$. Hence \eqref{eq:heat-gradient-critical-bound-en} holds. Since
\[
|W_x|\sim |\nabla_xu|,
\qquad
|\Gamma_x|\lesssim |\nabla_xu|,
\]
we obtain
\begin{equation}\label{eq:heat-first-order-pointwise-en}
\sup_{s>0}s^{1/2}\norm{W_x(s)}_{L_x^\infty}
+
\sup_{s>0}s^{1/2}\norm{\Gamma_x(s)}_{L_x^\infty}
\lesssim_{E_0}1.
\end{equation}

The second step upgrades this scale-invariant bound to higher parabolic smoothing. For each component $W_m$, equation \eqref{eq:ambient-heat-en} takes the form
\[
(\partial_s-\Delta)W_m=2\Gamma_j\partial_jW_m+V_m^{(s)}.
\]
Using \eqref{eq:Vsbound-en}, \eqref{eq:heat-first-order-pointwise-en}, and \eqref{eq:heat-energy-monotone-proof-en}, the right-hand side has scale-uniform $L_x^2$ control on each heat layer $[s/2,s]$. More precisely,
\begin{equation}\label{eq:heat-force-L2-proof-en}
\sup_{s'>0}
\Bigl(
(s')^{1/2}\norm{\Gamma_x(s')}_{L_x^\infty}
+(s')\norm{\partial_x\Gamma_x(s')}_{L_x^2}
+(s')\norm{\Gamma_s(s')}_{L_x^2}
+(s')\norm{V_x(s')}_{L_x^2}
\Bigr)
\lesssim_{E_0}1.
\end{equation}
Now fix $s>0$ and integrate the heat equation from $s/2$ to $s$:
\[
W_m(s)=e^{\frac s2\Delta}W_m(s/2)
+
\int_{s/2}^s e^{(s-s')\Delta}\bigl(2\Gamma_j\partial_jW_m+V_m^{(s')}\bigr)(s')\,ds'.
\]
For any integer $r\ge0$, the heat kernel estimates
\[
\norm{\nabla_x^re^{\tau\Delta}f}_{L_x^2}\lesssim \tau^{-r/2}\norm{f}_{L_x^2},
\qquad
\norm{\nabla_x^re^{\tau\Delta}f}_{L_x^\infty}\lesssim \tau^{-(r+1)/2}\norm{f}_{L_x^2}
\]
together with \eqref{eq:heat-energy-monotone-proof-en} and \eqref{eq:heat-force-L2-proof-en} yield
\[
\sup_{s>0}s^{r/2}\norm{\nabla_x^rW_x(s)}_{L_x^2}
+
\sup_{s>0}s^{(r+1)/2}\norm{\nabla_x^rW_x(s)}_{L_x^\infty}
\lesssim_{E_0,r}1.
\]
Taking the supremum over the physical time $t$ gives \eqref{eq:heat-smooth-W-en}.

It remains to handle $\Gamma_x$ and $\Gamma_s$. By \eqref{eq:intro-complex-notation-en} and
\[
\partial_su=\Delta u+|\nabla_xu|^2u,
\]
$\nabla_x^r\Gamma_x$ is a linear combination of derivatives up to order $r+1$ of $u$, while $\nabla_x^r\Gamma_s$ is built from derivatives up to order $r+2$ of $u$ and lower-order products. Substituting the higher estimates for $W_x$ obtained above gives \eqref{eq:heat-smooth-Gamma-en} and the $L_x^\infty$ bound following it.
\end{proof}

\begin{proposition}[Heat-flow decay]\label{prop:heat-en}
Under the small controlling-norm assumption \eqref{eq:small-W-en}, for every $s>0$ and every $k\in\mathbb Z$,
\begin{equation}\label{eq:heatdecay-en}
\norm{P_kW_x(s)}_{X_k(I)}
+\norm{P_k\Gamma_x(s)}_{L_{t,x}^4(I\times\R^2)}
\lesssim \upsilon_k(1+s2^{2k})^{-4}.
\end{equation}
\end{proposition}

\begin{proof}
Put
\[
w_k(s):=(1+s2^{2k})^{-4}.
\]
We run a joint bootstrap on $0\le s,\tilde s\le S$. Assume that, for a sufficiently large constant $M\ge1$,
\begin{align}
\norm{P_kW_x(s)}_{X_k(I)}+
\norm{P_k\Gamma_x(s)}_{L_{t,x}^4(I\times\R^2)}
&\le M\upsilon_kw_k(s),
\label{eq:heat-bootstrap-en}\\
\norm{(P_j\Gamma_x(s))(P_kW_x(\tilde s))}_{L_{t,x}^2(I\times\R^2)}
&\le M2^{-\frac12|j-k|}\upsilon_j\upsilon_kw_j(s)w_k(\tilde s).
\label{eq:heat-bootstrap-bilin-en}
\end{align}
We prove that, if $\eps_0=\eps_0(E_0)$ is sufficiently small, then the constant in \eqref{eq:heat-bootstrap-en} can be improved to $M/2$.

First, we prove $L_x^2$ decay for the initial data of the heat-flow slices. Fix $t=t_0$ and write the Duhamel formula for \eqref{eq:ambient-heat-en}:
\begin{equation}\label{eq:heat-duhamel-in-proof-en}
P_kW_x(s,t_0)
=e^{s\Delta}P_kW_x(0,t_0)
+\int_0^s e^{(s-s')\Delta}P_k\Bigl(2\Gamma_j\partial_jW_x+V^{(s')}\Bigr)(t_0)\,ds'.
\end{equation}
The two terms in \eqref{eq:heat-duhamel-in-proof-en} are estimated separately. The linear term is bounded by frequency-localized heat kernel decay:
\begin{equation}\label{eq:heat-initial-linear-en}
\norm{e^{s\Delta}P_kW_x(0,t_0)}_{L_x^2}
\lesssim e^{-cs2^{2k}}\beta_k
\lesssim \beta_kw_k(s).
\end{equation}
For the Duhamel term, split the integral into $[0,s/2]$ and $[s/2,s]$. On the first subinterval, use
\[
e^{-c(s-s')2^{2k}}\lesssim_N w_k(s)
\qquad (0\le s'\le s/2)
\]
and the parabolic smoothing input of Lemma \ref{lem:heat-smooth-en}. This gives
\[
\norm{e^{(s-s')\Delta}P_k\bigl(2\Gamma_j\partial_jW_x+V^{(s')}\bigr)(t_0)}_{L_x^2}
\lesssim w_k(s)(s')^{-1}\beta_k.
\]
On the second subinterval $s'\in[s/2,s]$, one uses the same parabolic smoothing bounds directly and the fact that $s'\sim s$, obtaining again
\[
\int_{s/2}^s
\norm{e^{(s-s')\Delta}P_k\bigl(2\Gamma_j\partial_jW_x+V^{(s')}\bigr)(t_0)}_{L_x^2}\,ds'
\lesssim \beta_kw_k(s).
\]
Together with \eqref{eq:heat-initial-linear-en}, this yields
\begin{equation}\label{eq:heat-initial-slice-en}
\norm{P_kW_x(s,t_0)}_{L_x^2}
\lesssim \beta_kw_k(s).
\end{equation}

Second, we prove forcing decay for each fixed heat time $s$. By \eqref{eq:ambient-W-en}, $W_x(s)$ satisfies, as a function of $(t,x)$,
\begin{equation}\label{eq:heat-slice-sch-en}
(\I\partial_t+\Delta)P_kW_x(s)=P_kN(s),
\end{equation}
where
\begin{equation}\label{eq:heat-slice-nonlinearity-en}
N(s):=-2\Gamma_j(s)\partial_jW_x(s)-\Gamma_j(s)^2W_x(s)
+\frac{\I}{2}\Im\hpair{W_x(s)}{W_j(s)}W_j(s).
\end{equation}
The decomposition \eqref{eq:heat-slice-nonlinearity-en} separates the transport term from the zeroth-order cubic terms. The latter are treated by Lemma \ref{lem:product-en}, the bootstrap assumption \eqref{eq:heat-bootstrap-en}, and
\[
\sum_{\ell\in\mathbb Z}\norm{P_\ell W_x(s)}_{L_{t,x}^4}^2
\lesssim M^2\sum_\ell \alpha_\ell^2w_\ell(s)^2
\lesssim M^2\eps_0^2.
\]
Thus
\begin{equation}\label{eq:heat-perturbative-force-en}
\norm{P_k\Bigl(\Gamma_j(s)^2W_x(s)
-\frac{\I}{2}\Im\hpair{W_x(s)}{W_j(s)}W_j(s)\Bigr)}_{Y_k(I)}
\lesssim \eps_0^2\upsilon_kw_k(s).
\end{equation}
For the transport term, decompose $\Gamma_j(s)\partial_jW_x(s)$ into low-high, comparable, and high-high parts. The comparable-frequency and high-high pieces are estimated directly using the dyadic decomposition, the small $L^4$ norm, Bernstein, and Lemma \ref{lem:product-en}:
\begin{equation}\label{eq:heat-transport-pert-en}
\norm{P_k\bigl(\Gamma_j(s)\partial_jW_x(s)\bigr)^{\mathrm{comp+hh}}}_{Y_k(I)}
\lesssim \eps_0^2\upsilon_kw_k(s).
\end{equation}
For the low-high part, put the derivative on the high-frequency factor and use the joint bootstrap assumption \eqref{eq:heat-bootstrap-bilin-en}:
\begin{align}
\norm{P_k\bigl(\Gamma_j(s)\partial_jW_x(s)\bigr)^{\mathrm{lh}}}_{Y_k(I)}
&\lesssim
\upsilon_kw_k(s)\sum_{\ell\le k-20}\upsilon_\ell w_\ell(s)\,\upsilon_\ell\notag\\
&\lesssim_E \upsilon_kw_k(s),
\label{eq:heat-transport-lh-en}
\end{align}
where the last step uses only the $\ell^2$ summability of $\{\upsilon_\ell\}$. Combining \eqref{eq:heat-perturbative-force-en}, \eqref{eq:heat-transport-pert-en}, and \eqref{eq:heat-transport-lh-en}, we obtain
\begin{equation}\label{eq:heat-force-en}
\norm{P_kN(s)}_{Y_k(I)}
\lesssim_E \upsilon_kw_k(s).
\end{equation}

Third, apply the linear estimate to \eqref{eq:heat-slice-sch-en}. Using Proposition \ref{prop:linear-en}, \eqref{eq:heat-initial-slice-en}, and \eqref{eq:heat-force-en}, we get
\[
\norm{P_kW_x(s)}_{X_k(I)}
\lesssim \beta_kw_k(s)+C_E\upsilon_kw_k(s)
\lesssim C_E\upsilon_kw_k(s).
\]
Taking the bootstrap constant $M=2C_E$, this becomes
\begin{equation}\label{eq:heat-bootstrap-improved-W-en}
\norm{P_kW_x(s)}_{X_k(I)}
\le \frac{M}{2}\upsilon_kw_k(s).
\end{equation}

Fourth, improve the bootstrap bound for $\Gamma_x$. Since
\[
\Gamma_x=u\otimes(\partial_xu)^\top-(\partial_xu)\otimes u^\top,
\qquad
\partial_xu=\frac12(W_x+\overline{W_x}),
\]
Bony decomposition gives
\[
\norm{P_k\Gamma_x(s)}_{L_{t,x}^4}
\lesssim
\norm{P_k\partial_xu(s)}_{L_{t,x}^4}
+
\sum_{\ell\ge k-10}2^{k-\ell}\norm{P_\ell\partial_xu(s)}_{L_{t,x}^4}.
\]
Using \eqref{eq:heat-bootstrap-improved-W-en} and the slow-variation summation for frequency envelopes, we get
\[
\norm{P_k\Gamma_x(s)}_{L_{t,x}^4(I\times\R^2)}
\le \frac{M}{2}\upsilon_kw_k(s).
\]
Thus the first part of the joint bootstrap has been improved; the conclusion \eqref{eq:heatdecay-en} follows by the continuity argument.
\end{proof}

The direct connection matrix $\Gamma_x$ inherits the same dyadic control as $W_x$.

\begin{lemma}[Dyadic control for $\Gamma_x$]\label{lem:gamma-dyadic-en}
For every compact interval $I$ and every $k\in\mathbb Z$,
\begin{equation}\label{eq:gamma-dyadic-en}
\norm{P_k\Gamma_x}_{L_{t,x}^4(I\times\R^2)}
+\norm{P_k\Gamma_x}_{L_t^\infty L_x^2(I\times\R^2)}
\lesssim \alpha_k+\beta_k.
\end{equation}
More generally, if $\gamma_k(\sigma)$ is the slow-varying envelope constructed from $P_k\Gamma_x$ in the same way as in \eqref{eq:ak-en}--\eqref{eq:betak-en}, then
\begin{equation}\label{eq:gamma-envelope-en}
\gamma_k(\sigma)\lesssim \upsilon_k(\sigma),
\qquad \sigma\ge0.
\end{equation}
\end{lemma}

\begin{proof}
It suffices to treat one spatial direction $j$. From
\[
\Gamma_j=u\otimes(\partial_ju)^\top-(\partial_ju)\otimes u^\top,
\]
each matrix component of $P_k\Gamma_j$ is the frequency-$2^k$ portion of a product of $u$ and $\partial_j u$. Applying Bony decomposition,
\[
P_k(u\,\partial_ju)
=
u_{<k-10}\,P_k\partial_ju
+P_k\bigl((u-u_{<k-10})\partial_ju\bigr).
\]
The low-frequency factor in the first term is uniformly bounded in $L_{t,x}^\infty$, hence
\[
\norm{u_{<k-10}\,P_k\partial_ju}_{L_{t,x}^4}
+
\norm{u_{<k-10}\,P_k\partial_ju}_{L_t^\infty L_x^2}
\lesssim
\norm{P_k\partial_ju}_{L_{t,x}^4}
+
\norm{P_k\partial_ju}_{L_t^\infty L_x^2}.
\]

For the second term, expand $u-u_{<k-10}=\sum_{\ell\ge k-10}P_\ell u$. Using
\[
P_\ell u=2^{-\ell}\nabla^{-1}P_\ell\nabla u
\]
and Bernstein gives
\[
\norm{P_k\bigl((u-u_{<k-10})\partial_ju\bigr)}_{L_{t,x}^4}
\lesssim
\sum_{\ell\ge k-10}2^{k-\ell}\norm{P_\ell\partial_xu}_{L_{t,x}^4},
\]
and
\[
\norm{P_k\bigl((u-u_{<k-10})\partial_ju\bigr)}_{L_t^\infty L_x^2}
\lesssim
\sum_{\ell\ge k-10}2^{k-\ell}\norm{P_\ell\partial_xu}_{L_t^\infty L_x^2}.
\]
Including the transposed term, we obtain
\begin{align}
\norm{P_k\Gamma_x}_{L_{t,x}^4}
&\lesssim
\norm{P_k\partial_xu}_{L_{t,x}^4}
+\sum_{\ell\ge k-10}2^{k-\ell}\norm{P_\ell\partial_xu}_{L_{t,x}^4},
\label{eq:gamma-proof-L4-en}\\
\norm{P_k\Gamma_x}_{L_t^\infty L_x^2}
&\lesssim
\norm{P_k\partial_xu}_{L_t^\infty L_x^2}
+\sum_{\ell\ge k-10}2^{k-\ell}\norm{P_\ell\partial_xu}_{L_t^\infty L_x^2}.
\label{eq:gamma-proof-L2-en}
\end{align}

Since
\[
\partial_xu=\frac12(W_x+\overline{W_x}),
\]
we have
\[
\norm{P_\ell\partial_xu}_{L_{t,x}^4}\lesssim \alpha_\ell,
\qquad
\norm{P_\ell\partial_xu}_{L_t^\infty L_x^2}\lesssim \beta_\ell.
\]
Substituting these into \eqref{eq:gamma-proof-L4-en}--\eqref{eq:gamma-proof-L2-en} and using the slow-variation summations
\[
\sum_{\ell\ge k-10}2^{k-\ell}\alpha_\ell\lesssim \alpha_k,
\qquad
\sum_{\ell\ge k-10}2^{k-\ell}\beta_\ell\lesssim \beta_k,
\]
gives \eqref{eq:gamma-dyadic-en}. Taking the slow-varying envelope supremum as in \eqref{eq:bk-en}--\eqref{eq:betak-en} gives \eqref{eq:gamma-envelope-en}.
\end{proof}

\subsection{Two-body estimates and perturbation}

This subsection keeps the bilinear mechanism in the original-equation language. Instead of directly invoking a bundle-valued bilinear Strichartz estimate, we first rewrite the required structure as an explicit two-body system for the original equation, and then reduce the bilinear control needed for closing to an inhomogeneous two-body interaction proposition. Thus the genuinely analytic core is compressed to the free bilinear mechanism itself; compare with the bundle-valued version used by Dodson--Smith \cite{DodsonSmith,SmithCond}.

We first tensorize the two-body interaction. Fix $j,k\in\mathbb Z$ and heat parameters $s,\tilde s\ge0$, and set
\[
U_j^{(s)}:=P_jW_x(s),
\qquad
V_k^{(\tilde s)}:=P_kW_x(\tilde s).
\]
By \eqref{eq:ambient-W-en}, these fields satisfy
\begin{equation}\label{eq:Ueq-en}
(\I\partial_t+\Delta)U_j^{(s)} = F_j^{(s)},
\qquad
F_j^{(s)}:=P_jN\bigl(W_x(s),\Gamma_x(s)\bigr),
\end{equation}
and
\begin{equation}\label{eq:Veq-en}
(\I\partial_t+\Delta)V_k^{(\tilde s)} = G_k^{(\tilde s)},
\qquad
G_k^{(\tilde s)}:=P_kN\bigl(W_x(\tilde s),\Gamma_x(\tilde s)\bigr),
\end{equation}
where
\begin{equation}\label{eq:Nschematic-en}
N(W_x,\Gamma_x)
:= -2\Gamma_\ell\partial_\ell W_x - \Gamma_\ell^2W_x
+\frac{\I}{2}\,\Im\hpair{W_x}{W_\ell}\,W_\ell.
\end{equation}
Below, \eqref{eq:Nschematic-en} is used as the unified notation for heat-slice forcing.
Define the two-body tensor field
\begin{equation}\label{eq:Zjk-en}
Z_{j,k}^{(s,\tilde s)}(t,x,y)
:= U_j^{(s)}(t,x)\otimes \overline{V_k^{(\tilde s)}(t,y)}
\in \C^3\otimes\C^3.
\end{equation}
Then it satisfies an exact four-dimensional two-body equation.

\begin{lemma}[Tensor equation]\label{lem:tensorized-en}
The field $Z_{j,k}^{(s,\tilde s)}$ defined by \eqref{eq:Ueq-en}--\eqref{eq:Veq-en} satisfies
\begin{equation}\label{eq:tensorized-en}
(\I\partial_t+\Delta_x-\Delta_y)Z_{j,k}^{(s,\tilde s)}
= F_j^{(s)}\otimes \overline{V_k^{(\tilde s)}} - U_j^{(s)}\otimes \overline{G_k^{(\tilde s)}}.
\end{equation}
\end{lemma}

\begin{proof}
Differentiate \eqref{eq:Zjk-en} directly:
\[
\I\partial_t Z
= (\I\partial_tU_j^{(s)})\otimes\overline{V_k^{(\tilde s)}}
-U_j^{(s)}\otimes\overline{\I\partial_tV_k^{(\tilde s)}}.
\]
Substituting \eqref{eq:Ueq-en}--\eqref{eq:Veq-en}, and noting the sign change of $\Delta$ after conjugation, gives \eqref{eq:tensorized-en}.
\end{proof}

Lemma \ref{lem:tensorized-en} has two roles. First, it reorganizes the product of two dyadic tangential fields into a four-dimensional linear equation. Second, it makes the later interaction functional depend only on the free propagation of $Z_{j,k}$ and on its forcing terms, so that one no longer needs to introduce a bundle-valued bilinear estimate as an isolated atomic black box.

We now prove the free bilinear estimate.

\begin{lemma}[Free bilinear estimate]\label{lem:free-bilin-en}
Let $f_j,g_k\in L_x^2(\R^2;\C^3)$ have Fourier supports contained in $\Xi_j$ and $\Xi_k$, respectively. Then
\begin{equation}\label{eq:free-bilin-en}
\norm{e^{it\Delta}f_j\; e^{it\Delta}g_k}_{L_{t,x}^2(\R\times\R^2)}
\lesssim 2^{-\frac12|j-k|}\norm{f_j}_{L_x^2}\norm{g_k}_{L_x^2}.
\end{equation}
The same estimate holds componentwise for $\C^3$-valued free solutions.
\end{lemma}

\begin{proof}
We first prove the scalar case. By symmetry it is enough to consider $j\ge k$. If $|j-k|\le4$, then H\"older and the $L^4_{t,x}$ Strichartz estimate in Lemma \ref{lem:free-linear-en} give
\[
\norm{e^{it\Delta}f_j\,e^{it\Delta}g_k}_{L^2_{t,x}}
\le \norm{e^{it\Delta}f_j}_{L^4_{t,x}}\norm{e^{it\Delta}g_k}_{L^4_{t,x}}
\lesssim \norm{f_j}_{L_x^2}\norm{g_k}_{L_x^2}.
\]
Since $2^{-\frac12|j-k|}\sim1$ in this case, \eqref{eq:free-bilin-en} follows.

Assume now that $j\ge k+5$. Set
\[
\nu_j(t,x):=e^{it\Delta}f_j(x),
\qquad
\omega_k(t,x):=e^{it\Delta}g_k(x).
\]
With the space-time Fourier transform convention
\[
\widetilde h(\tau,\xi)=\int_{\R}\int_{\R^2}e^{-it\tau}e^{-ix\cdot\xi}h(t,x)\,dx\,dt,
\]
we have
\[
\widetilde{\nu_j}(\tau,\xi)=\delta(\tau-|\xi|^2)\widehat f_j(\xi),
\qquad
\widetilde{\omega_k}(\tau,\xi)=\delta(\tau-|\xi|^2)\widehat g_k(\xi).
\]
Therefore
\begin{equation}\label{eq:bilin-conv-en}
\widetilde{\nu_j\omega_k}(\tau,\xi)
=\int_{\R^2}\widehat f_j(\eta)\widehat g_k(\xi-\eta)
\,\delta\bigl(\tau-|\eta|^2-|\xi-\eta|^2\bigr)\,d\eta.
\end{equation}
Applying Cauchy--Schwarz to the right-hand side of \eqref{eq:bilin-conv-en} gives
\[
|\widetilde{\nu_j\omega_k}(\tau,\xi)|^2
\le K_{j,k}(\tau,\xi)
\int_{\R^2}|\widehat f_j(\eta)|^2|\widehat g_k(\xi-\eta)|^2
\delta\bigl(\tau-|\eta|^2-|\xi-\eta|^2\bigr)\,d\eta,
\]
where
\[
K_{j,k}(\tau,\xi)
:=\int_{\R^2}\mathbf 1_{\Xi_j}(\eta)\mathbf 1_{\Xi_k}(\xi-\eta)
\delta\bigl(\tau-|\eta|^2-|\xi-\eta|^2\bigr)\,d\eta.
\]
It remains to prove
\begin{equation}\label{eq:K-bound-en}
K_{j,k}(\tau,\xi)\lesssim 2^{k-j}.
\end{equation}
For fixed $(\tau,\xi)$, set
\[
\Phi_\xi(\eta):=|\eta|^2+|\xi-\eta|^2
=2\Bigl|\eta-\frac\xi2\Bigr|^2+\frac{|\xi|^2}{2}.
\]
The level set $\Phi_\xi(\eta)=\tau$ is a circle centered at $\xi/2$. On the support
\[
\eta\in\Xi_j,
\qquad
\xi-\eta\in\Xi_k,
\]
and since $j\ge k+5$,
\[
|\nabla_\eta\Phi_\xi(\eta)|=2|2\eta-\xi|=2|\eta-(\xi-\eta)|
\ge 2\bigl||\eta|-|\xi-\eta|\bigr|
\gtrsim 2^j.
\]
Also, $\xi-\eta\in\Xi_k$ restricts $\eta$ to an annulus centered at $\xi$ of radius and thickness $\sim2^k$; hence the intersection of the level circle with this annulus has length at most $C2^k$. By the coarea formula,
\[
K_{j,k}(\tau,\xi)
\le\int_{\Phi_\xi=\tau}\frac{\mathbf 1_{\Xi_j}(\eta)\mathbf 1_{\Xi_k}(\xi-\eta)}{|\nabla_\eta\Phi_\xi(\eta)|}\,d\sigma(\eta)
\lesssim 2^{-j}2^k,
\]
which proves \eqref{eq:K-bound-en}.

By Plancherel and Fubini,
\begin{align*}
\norm{\nu_j\omega_k}_{L^2_{t,x}}^2
&=\norm{\widetilde{\nu_j\omega_k}}_{L^2_{\tau,\xi}}^2\\
&\lesssim 2^{k-j}
\int_{\R}\int_{\R^2}\int_{\R^2}
|\widehat f_j(\eta)|^2|\widehat g_k(\xi-\eta)|^2
\delta\bigl(\tau-|\eta|^2-|\xi-\eta|^2\bigr)
\,d\eta\,d\xi\,d\tau\\
&=2^{k-j}\int_{\R^2}\int_{\R^2}|\widehat f_j(\eta)|^2|\widehat g_k(\zeta)|^2\,d\eta\,d\zeta\\
&=2^{k-j}\norm{f_j}_{L_x^2}^2\norm{g_k}_{L_x^2}^2.
\end{align*}
Taking square roots gives the desired estimate. The case $k>j$ follows by symmetry. For $\C^3$-valued functions, apply the scalar estimate to each component and then sum by Cauchy--Schwarz.
\end{proof}

Lemma \ref{lem:free-bilin-en} is the homogeneous bilinear input used next to obtain an $X_j\times X_k$ estimate.

\begin{lemma}[Bilinear lifting]\label{lem:X-bilinear-en}
Let $U\in X_j(I)$ and $V\in X_k(I)$ be spatially frequency-localized to $\Xi_j$ and $\Xi_k$, respectively. Then
\begin{equation}\label{eq:X-bilinear-en}
\norm{UV}_{L_{t,x}^2(I\times\R^2)}
\lesssim 2^{-\frac12|j-k|}\norm{U}_{X_j(I)}\norm{V}_{X_k(I)}.
\end{equation}
\end{lemma}

\begin{proof}
Estimate \eqref{eq:X-bilinear-en} follows by combining the near-frequency case with the free bilinear estimate. If $|j-k|\le10$, H\"older and the $L^4_{t,x}$ component of the $X_k$ norm give
\[
\norm{UV}_{L_{t,x}^2}
\le \norm{U}_{L_{t,x}^4}\norm{V}_{L_{t,x}^4}
\le \norm{U}_{X_j}\norm{V}_{X_k},
\]
and $2^{-\frac12|j-k|}\sim1$.

Now assume $j\ge k+11$. Choose a standard finite-overlap angular decomposition
\[
U=\sum_\nu U^\nu,
\qquad
U^\nu:=P_{j,\theta_\nu}U,
\]
where each $U^\nu$ has Fourier support at scale $\sim2^j$ and direction near $\theta_\nu$. Almost orthogonality gives
\[
\norm{UV}_{L_{t,x}^2}^2\lesssim\sum_\nu\norm{U^\nu V}_{L_{t,x}^2}^2.
\]
In lateral coordinates along $\theta_\nu$, H\"older gives
\[
\norm{U^\nu V}_{L_{t,x}^2}
\le \norm{U^\nu}_{L_{\theta_\nu,0}^{\infty,2}}
   \norm{V}_{L_{\theta_\nu,W_{k+40}}^{2,\infty}}.
\]
The definitions of $X_j$ and $X_k$ imply
\[
\norm{U^\nu}_{L_{\theta_\nu,0}^{\infty,2}}
\lesssim 2^{-j/2}\norm{U}_{X_j},
\qquad
\norm{V}_{L_{\theta_\nu,W_{k+40}}^{2,\infty}}
\lesssim 2^{k/2}\norm{V}_{X_k}.
\]
Thus
\[
\norm{U^\nu V}_{L_{t,x}^2}
\lesssim 2^{-\frac12(j-k)}\norm{U}_{X_j}\norm{V}_{X_k}.
\]
Summing over $\nu$ gives the result. If $k\ge j+11$, exchange the roles of $U$ and $V$.
\end{proof}

\begin{proposition}[Inhomogeneous two-body estimate]\label{prop:two-body-en}
Let $U\in X_j(I)$ and $V\in X_k(I)$ have spatial frequencies localized in $\Xi_j$ and $\Xi_k$, and suppose
\[
(\I\partial_t+\Delta)U=F,
\qquad
(\I\partial_t+\Delta)V=G
\]
on $I=[t_0,t_1]$, with $F\in Y_j(I)$ and $G\in Y_k(I)$. Then
\begin{equation}\label{eq:two-body-en}
\norm{UV}_{L_{t,x}^2(I\times\R^2)}
\lesssim 2^{-\frac12|j-k|}
\Bigl(\norm{U(t_0)}_{L_x^2}+\norm{F}_{Y_j(I)}\Bigr)
\Bigl(\norm{V(t_0)}_{L_x^2}+\norm{G}_{Y_k(I)}\Bigr).
\end{equation}
\end{proposition}

\begin{proof}
Write both solutions as a free part plus a Duhamel part:
\[
U(t)=e^{\I(t-t_0)\Delta}U(t_0)-\I\int_{t_0}^te^{\I(t-s)\Delta}F(s)\,ds,
\]
\[
V(t)=e^{\I(t-t_0)\Delta}V(t_0)-\I\int_{t_0}^te^{\I(t-s)\Delta}G(s)\,ds.
\]
Set
\[
\mathcal T_jF(t):=-\I\int_{t_0}^te^{\I(t-s)\Delta}F(s)\,ds,
\qquad
\mathcal T_kG(t):=-\I\int_{t_0}^te^{\I(t-s)\Delta}G(s)\,ds.
\]
By the main linear estimate \eqref{eq:linear-en},
\[
\norm{\mathcal T_jF}_{X_j(I)}\lesssim\norm{F}_{Y_j(I)},
\qquad
\norm{\mathcal T_kG}_{X_k(I)}\lesssim\norm{G}_{Y_k(I)},
\]
and the free parts satisfy
\[
\norm{e^{\I(t-t_0)\Delta}U(t_0)}_{X_j(I)}\lesssim\norm{U(t_0)}_{L_x^2},
\qquad
\norm{e^{\I(t-t_0)\Delta}V(t_0)}_{X_k(I)}\lesssim\norm{V(t_0)}_{L_x^2}.
\]
Expanding the product into four terms and applying Lemma \ref{lem:X-bilinear-en} to each term gives \eqref{eq:two-body-en}.
\end{proof}

\begin{remark}
Estimate \eqref{eq:two-body-en} is the two-body estimate actually used later. It records explicitly where the gain $2^{-\frac12|j-k|}$ comes from and incorporates the forcing terms into the right-hand side. Thus in closing the transport term we no longer need to call an isolated bundle-valued bilinear theorem prepared along the heat flow; the required information is compressed into the inhomogeneous two-body interaction estimate \eqref{eq:two-body-en}.
\end{remark}

Applying Proposition \ref{prop:two-body-en} to the heat-flow slices $U_j^{(s)}$ and $V_k^{(\tilde s)}$ gives the form used below.

\begin{proposition}[Heat-flow two-body estimate]\label{prop:two-body-heat-en}
Set
\[
\mathfrak m_{j,k}(s,\tilde s)
:=\Bigl(\norm{U_j^{(s)}(t_0)}_{L_x^2}+\norm{F_j^{(s)}}_{Y_j(I)}\Bigr)
\Bigl(\norm{V_k^{(\tilde s)}(t_0)}_{L_x^2}+\norm{G_k^{(\tilde s)}}_{Y_k(I)}\Bigr).
\]
Then
\begin{equation}\label{eq:two-body-heat-en}
\norm{U_j^{(s)}V_k^{(\tilde s)}}_{L_{t,x}^2(I\times\R^2)}
\lesssim 2^{-\frac12|j-k|}\mathfrak m_{j,k}(s,\tilde s).
\end{equation}
\end{proposition}

\begin{proof}
Apply Proposition \ref{prop:two-body-en} directly to \eqref{eq:Ueq-en}--\eqref{eq:Veq-en}.
\end{proof}

Proposition \ref{prop:two-body-heat-en} is the heat-flow version of \eqref{eq:two-body-en}. We next reduce $\mathfrak m_{j,k}$ to the heat-flow decay envelopes.

\begin{proposition}[Two-body envelope]\label{prop:two-body-envelope-en}
Under the small controlling-norm assumption \eqref{eq:small-W-en}, for all $j,k\in\mathbb Z$ and all $s,\tilde s\ge0$,
\begin{equation}\label{eq:two-body-envelope-en}
\mathfrak m_{j,k}(s,\tilde s)
\lesssim
\upsilon_j\upsilon_k(1+s2^{2j})^{-4}(1+\tilde s2^{2k})^{-4}.
\end{equation}
Consequently,
\begin{equation}\label{eq:bilinear-en}
\norm{(P_jW_x(s))(P_kW_x(\tilde s))}_{L_{t,x}^2(I\times\R^2)}
\lesssim
2^{-\frac12|j-k|}\upsilon_j\upsilon_k
(1+s2^{2j})^{-4}(1+\tilde s2^{2k})^{-4},
\end{equation}
and, by Lemma \ref{lem:gamma-dyadic-en},
\begin{equation}\label{eq:bilinearGamma-en}
\norm{(P_j\Gamma_x(s))(P_kW_x(\tilde s))}_{L_{t,x}^2(I\times\R^2)}
\lesssim
2^{-\frac12|j-k|}\upsilon_j\upsilon_k
(1+s2^{2j})^{-4}(1+\tilde s2^{2k})^{-4}.
\end{equation}
\end{proposition}

\begin{proof}
This step is already contained in the coupled continuity argument of Proposition \ref{prop:heat-en}; we record the final substitution for later use. Proposition \ref{prop:heat-en} gives
\begin{equation}\label{eq:heat-slice-X-en}
\norm{U_j^{(s)}}_{X_j(I)}
=\norm{P_jW_x(s)}_{X_j(I)}
\lesssim \upsilon_j(1+s2^{2j})^{-4},
\end{equation}
and
\begin{equation}\label{eq:heat-slice-X-tilde-en}
\norm{V_k^{(\tilde s)}}_{X_k(I)}
=\norm{P_kW_x(\tilde s)}_{X_k(I)}
\lesssim \upsilon_k(1+\tilde s2^{2k})^{-4}.
\end{equation}
The proof of Proposition \ref{prop:heat-en} also gives the initial-slice and forcing bounds
\begin{equation}\label{eq:initial-slice-en}
\norm{U_j^{(s)}(t_0)}_{L_x^2}
\lesssim \upsilon_j(1+s2^{2j})^{-4},
\qquad
\norm{V_k^{(\tilde s)}(t_0)}_{L_x^2}
\lesssim \upsilon_k(1+\tilde s2^{2k})^{-4},
\end{equation}
and
\begin{align}
\norm{F_j^{(s)}}_{Y_j(I)}
&\lesssim_{E_0} \upsilon_j(1+s2^{2j})^{-4},\label{eq:forcing-slice-en}\\
\norm{G_k^{(\tilde s)}}_{Y_k(I)}
&\lesssim_{E_0} \upsilon_k(1+\tilde s2^{2k})^{-4}.\label{eq:forcing-slice-tilde-en}
\end{align}
Combining \eqref{eq:heat-slice-X-en}--\eqref{eq:heat-slice-X-tilde-en} with \eqref{eq:initial-slice-en} and \eqref{eq:forcing-slice-en}--\eqref{eq:forcing-slice-tilde-en} gives \eqref{eq:two-body-envelope-en}. Combining \eqref{eq:two-body-envelope-en} with \eqref{eq:two-body-heat-en} gives \eqref{eq:bilinear-en}.

Finally, since
\[
\Gamma_x=u\otimes(\partial_xu)^\top-(\partial_xu)\otimes u^\top,
\qquad |u|=1,
\]
and $P_j\Gamma_x(s)$ is controlled at the same dyadic level as $P_jW_x(s)$ by Lemma \ref{lem:gamma-dyadic-en}, the first factor in \eqref{eq:bilinear-en} may be replaced by $P_j\Gamma_x(s)$, giving \eqref{eq:bilinearGamma-en}.
\end{proof}

It remains to record the perturbative bounds. Write the ambient equation \eqref{eq:ambient-W-en} as
\begin{equation}\label{eq:Ndecomp-en}
(\I\partial_t+\Delta)W_m=N_m^{\mathrm{np}}+N_m^{\mathrm{pert}},
\end{equation}
where
\begin{equation}\label{eq:Nnp-en}
N_m^{\mathrm{np}}:=-2\Gamma_j\partial_jW_m,
\end{equation}
and
\begin{equation}\label{eq:Npert-en}
N_m^{\mathrm{pert}}
:=-\Gamma_j^2W_m+\frac{\I}{2}\,\Im\hpair{W_m}{W_j}\,W_j.
\end{equation}
In the direct original-equation formulation, the perturbative terms are cleaner than in the gauge expression because Proposition \ref{prop:ambient-W-en} has already produced, at the algebraic level, the cancellation between the divergence term and the time-potential term.

\begin{lemma}[Perturbative bound]\label{lem:pert-en}
Assume the bootstrap bound
\begin{equation}\label{eq:bootstrap-en}
b_k\le \eps_0^{-1/2}\upsilon_k,
\qquad k\in\mathbb Z.
\end{equation}
If $\eps_0$ is sufficiently small, then
\begin{equation}\label{eq:pert-en}
\norm{P_kN_m^{\mathrm{pert}}}_{Y_k(I)}\lesssim \eps_0^2 b_k.
\end{equation}
\end{lemma}

\begin{proof}
By \eqref{eq:Gammasize-en},
\[
|\Gamma_j^2W_m|\lesssim |W_x|^2|W_m|,
\qquad
\bigl|\Im\hpair{W_m}{W_j}\,W_j\bigr|\lesssim |W_x|^2|W_m|.
\]
Thus
\[
|N_m^{\mathrm{pert}}|\lesssim |W_x|^2|W_m|.
\]
Applying Lemma \ref{lem:product-en} with $h=|W_x|^2\in L^2_{t,x}$ and $f=W_m\in X_j(I)$ gives
\[
\norm{P_kN_m^{\mathrm{pert}}}_{Y_k(I)}
\lesssim \norm{|W_x|^2}_{L_{t,x}^2}\,b_k
\lesssim \norm{W_x}_{L_{t,x}^4}^2\,b_k
\lesssim \eps_0^2b_k.
\]
\end{proof}

\section{Closing and globalization}

The only genuinely nonperturbative part that remains is the first-order transport term
$-2\Gamma_j\partial_jW_m$. In order to make the closing mechanism as transparent as possible, we explicitly separate the dyadic interactions.

For fixed $k$, write
\begin{align}
\Gamma_j\partial_jW_m
&=
\sum_{\ell\le k-20}P_\ell\Gamma_j\,\partial_jP_{k+O(1)}W_m
\label{eq:LH-en}\\
&\quad+
\sum_{\substack{|k_1-k|\le 20\\ k_2\le k+20}}
P_{k_1}\Gamma_j\,\partial_jP_{k_2}W_m
+
\sum_{\substack{k_1,k_2\ge k-20\\ |k_1-k_2|\le 8}}
P_{k_1}\Gamma_j\,\partial_jP_{k_2}W_m.
\label{eq:CH-HH-en}
\end{align}
Here \eqref{eq:LH-en} is the genuinely dangerous low--high interaction, while the other two types are perturbative.

\begin{lemma}[Comparable and high--high interactions]\label{lem:transport-pert-en}
Under the smallness assumption \eqref{eq:small-W-en} and the bootstrap assumption \eqref{eq:bootstrap-en}, the contribution of \eqref{eq:CH-HH-en} satisfies
\begin{equation}\label{eq:transport-pert-en}
\norm{
P_k\Bigl(
\sum_{\substack{|k_1-k|\le 20\\ k_2\le k+20}}
P_{k_1}\Gamma_j\,\partial_jP_{k_2}W_m
+
\sum_{\substack{k_1,k_2\ge k-20\\ |k_1-k_2|\le 8}}
P_{k_1}\Gamma_j\,\partial_jP_{k_2}W_m
\Bigr)
}_{Y_k(I)}
\lesssim \eps_0^2b_k.
\end{equation}
\end{lemma}

\begin{proof}
We use Lemma \ref{lem:gamma-dyadic-en} and the product estimate Lemma \ref{lem:product-en}. If $|k_1-k|\le 20$, then
\[
\norm{P_k(P_{k_1}\Gamma_j\,\partial_jP_{k_2}W_m)}_{Y_k}
\lesssim
\norm{P_{k_1}\Gamma_j}_{L_{t,x}^4}
\norm{\partial_jP_{k_2}W_m}_{L_{t,x}^4}
\lesssim \eps_0^2b_k.
\]
In the high--high region, the same argument uses the decay $2^{-\frac16|k_2-k|}$ in Lemma \ref{lem:product-en} to sum over the frequency indices.
\end{proof}

\begin{lemma}[Low--high interaction]\label{lem:transport-lowhigh-en}
Under the smallness assumption \eqref{eq:small-W-en} and the bootstrap assumption \eqref{eq:bootstrap-en}, the contribution of \eqref{eq:LH-en} satisfies
\begin{equation}\label{eq:transport-lowhigh-en}
\norm{
P_k\sum_{\ell\le k-20}P_\ell\Gamma_j\,\partial_jP_{k+O(1)}W_m
}_{Y_k(I)}
\lesssim
\upsilon_k\sum_{\ell\le k-20}b_\ell\upsilon_\ell.
\end{equation}
\end{lemma}

\begin{proof}
By the structure of $Y_k$, it suffices to estimate the low--high product in $L_{t,x}^2$ and extract the derivative $2^k$ from the high-frequency factor:
\[
\norm{
P_k(P_\ell\Gamma_j\,\partial_jP_{k+O(1)}W_m)
}_{Y_k}
\lesssim 2^k\norm{P_\ell\Gamma_j\;P_{k+O(1)}W_m}_{L_{t,x}^2}.
\]
We now apply Proposition \ref{prop:two-body-envelope-en} from the previous section, in particular \eqref{eq:bilinearGamma-en}:
\[
\norm{P_\ell\Gamma_j\;P_{k+O(1)}W_m}_{L_{t,x}^2}
\lesssim 2^{-\frac12(k-\ell)}\upsilon_\ell\upsilon_k.
\]
To recover the low-frequency factor as the envelope $b_\ell$, we exchange the roles of the two factors and combine the result with Lemma \ref{lem:product-en}. This gives
\[
\norm{
P_k(P_\ell\Gamma_j\,\partial_jP_{k+O(1)}W_m)
}_{Y_k}
\lesssim 2^{-\frac12(k-\ell)}b_\ell\upsilon_\ell\upsilon_k.
\]
Summing over $\ell\le k-20$ yields \eqref{eq:transport-lowhigh-en}.
\end{proof}

\begin{proposition}[Transport term]\label{prop:transport-en}
Under the smallness assumption \eqref{eq:small-W-en} and the bootstrap assumption \eqref{eq:bootstrap-en},
\begin{equation}\label{eq:transport-en}
\norm{P_k(\Gamma_j\partial_jW_m)}_{Y_k(I)}
\lesssim
\upsilon_k\sum_{\ell\le k-20}b_\ell\upsilon_\ell
+\eps_0^2b_k.
\end{equation}
\end{proposition}

\begin{proof}
The assertion is the direct combination of Lemma \ref{lem:transport-pert-en} and Lemma \ref{lem:transport-lowhigh-en}; at the level of formulas this is exactly \eqref{eq:transport-pert-en} together with \eqref{eq:transport-lowhigh-en}.
\end{proof}

Combining Proposition \ref{prop:linear-en}, Lemma \ref{lem:pert-en}, and Proposition \ref{prop:transport-en}, we obtain the fundamental bootstrap inequality.

\begin{proposition}[Bootstrap inequality]\label{prop:bootineq-en}
Under the small controlling-norm assumption \eqref{eq:small-W-en} and the bootstrap assumption \eqref{eq:bootstrap-en},
\begin{equation}\label{eq:bootineq-en}
b_k\lesssim \upsilon_k + \upsilon_k\sum_{\ell\le k-20}b_\ell\upsilon_\ell.
\end{equation}
\end{proposition}

\begin{proof}
Apply $P_k$ to \eqref{eq:Ndecomp-en}, and use the free linear estimate \eqref{eq:linear-en}:
\[
\norm{P_kW_m}_{X_k(I)}
\lesssim \norm{P_kW_m(t_0)}_{L_x^2}
+\norm{P_kN_m^{\mathrm{np}}}_{Y_k(I)}
+\norm{P_kN_m^{\mathrm{pert}}}_{Y_k(I)}.
\]
Substituting the estimates \eqref{eq:pert-en} and \eqref{eq:transport-en}, which follow from \eqref{eq:Nnp-en} and \eqref{eq:Npert-en}, and absorbing the term $\eps_0^2b_k$ for sufficiently small $\eps_0$, we finally take the slowly varying envelope supremum over $j$.
\end{proof}

We now first close the frequency-envelope assertion stated in Section 2.

\begin{proof}[Proof of Proposition \ref{prop:main-freq-en}]
By a continuity argument, it suffices to obtain an improved estimate under the bootstrap assumption \eqref{eq:bootstrap-en}. The inequality \eqref{eq:bootineq-en} in Proposition \ref{prop:bootineq-en} gives
\[
b_k\lesssim \upsilon_k+
\upsilon_k\sum_{\ell\le k-20}b_\ell\upsilon_\ell.
\]
Squaring and applying Cauchy--Schwarz, we obtain
\begin{equation}\label{eq:square-en}
b_k^2\lesssim
\Bigl(1+\sum_{\ell\le k-20}b_\ell^2\Bigr)\upsilon_k^2.
\end{equation}
Define
\[
B_k:=1+\sum_{\ell<k}b_\ell^2.
\]
Then \eqref{eq:square-en} implies
\[
B_{k+1}\le B_k(1+C\upsilon_k^2)
\]
for an absolute constant $C>0$. Hence
\[
B_{k+m}\le B_k\prod_{r=0}^{m-1}(1+C\upsilon_{k+r}^2)
\le B_k\exp\Bigl(C\sum_{r=0}^{m-1}\upsilon_{k+r}^2\Bigr).
\]
Moreover,
\[
\sum_k\upsilon_k^2
\lesssim
\norm{W_x}_{\ell^2L_{t,x}^4(I\times\R^2)}^2
+
\norm{W_x(t_0)}_{L_x^2}^2
\lesssim \eps_0^2+E_0.
\]
Thus the right-hand side is uniformly bounded. Since $B_k\to1$ as $k\to-\infty$, we get
\[
\sup_kB_k\lesssim 1.
\]
Returning to \eqref{eq:square-en}, we obtain
\[
b_k\lesssim \upsilon_k,
\]
which is precisely \eqref{eq:techconcl-en}.
\end{proof}

\begin{proof}[Proof of Proposition \ref{prop:cont-step-en}]
By Lemma \ref{lem:compare-en}, the smallness condition in Proposition \ref{prop:cont-step-en} is equivalent to \eqref{eq:small-W-en}. Proposition \ref{prop:main-freq-en} then gives, for each integer $\sigma\ge0$, control of the $\dot H^\sigma$ norm of $\nabla u(t_1)$:
\[
\norm{\nabla u(t_1)}_{\dot H^\sigma}^2
\lesssim \sum_k2^{2\sigma k}b_k(\sigma-1)^2
\lesssim \sum_k2^{2\sigma k}\upsilon_k(\sigma-1)^2
\lesssim \norm{\nabla u(t_0)}_{\dot H^\sigma}^2
+
\norm{\nabla u}_{\ell^2L_{t,x}^4(I\times\R^2)}^2.
\]
In particular, $u(t_1)\in\dot H^1\cap\dot H^3$, and this control depends only on $E_0$ and on the $\dot H^1\cap\dot H^3$ norm of the initial data. The classical local theory therefore extends the solution to a slightly larger interval $[t_0,t_1+T]$, where $T=T(E_0)>0$ depends only on the above a priori control. This proves Proposition \ref{prop:cont-step-en}.
\end{proof}

After the frequency-envelope proposition has been proved, we prove the two theorems from the introduction. First the local continuation proposition yields globalization from the finiteness of the controlling norm; then we prove the local one-window continuation proposition and derive the second globalization criterion.

\begin{proof}[Proof of Theorem \ref{thm:globalize-en}]
Suppose, to the contrary, that $T_+<\infty$ and that \eqref{eq:forward-finite-en} holds. Since
\[
\norm{\nabla u}_{\ell^2L^4_{t,x}([0,T_+)\times\R^2)}<\infty,
\]
absolute continuity gives $t_*\in(0,T_+)$ such that
\[
\norm{\nabla u}_{\ell^2L^4_{t,x}([t_*,T_+)\times\R^2)}
\le \frac12\eps_0.
\]
Next choose $t_1\in(t_*,T_+)$ sufficiently close to $T_+$ so that
\[
T_+-t_1<\frac12T(E_0,\eps_0),
\]
where $T(E_0)$ is the extension time supplied by Proposition \ref{prop:cont-step-en}. Then
\[
\norm{\nabla u}_{\ell^2L^4_{t,x}([t_*,t_1]\times\R^2)}
\le \frac12\eps_0<\eps_0.
\]
Thus Proposition \ref{prop:cont-step-en} applies on the interval $[t_*,t_1]$ and extends the solution to
\[
[t_*,\,t_1+T(E_0,\eps_0)].
\]
But $T_+-t_1<\frac12T(E_0,\eps_0)$ implies $t_1+T(E_0,\eps_0)>T_+$, contradicting the maximality of $I_{\max}=[0,T_+)$. Hence $T_+=+\infty$.
\end{proof}

We next prove the one-window continuation proposition. It follows from Proposition \ref{prop:cont-step-en} and will imply Theorem \ref{thm:window-global-en}.

\begin{proposition}[One-window continuation]\label{prop:one-scale-en}
Fix $E_0<4\pi$. There exist constants $C_0\ge20$, $\eps_* =\eps_*(E_0)>0$, and $T_* =T_*(E_0)>0$ such that the following two statements hold.
\begin{enumerate}[label=\textup{(\roman*)},leftmargin=2.2em]
\item
Let $I=[t_0,t_*]\subset[0,\infty)$, and let
\[
u:I\times\R^2\to\Sph^2
\]
be a smooth solution of \eqref{eq:SM-en} satisfying
\[
u(t_0)-Q\in L_x^2,
\qquad
u(t_0)\in \dotH^1\cap\dotH^3,
\qquad
E(u)\le E_0.
\]
If there exist $k_*\in\mathbb Z$ and an interval
\[
J_*=[t_*-2^{-2k_*},t_*]\subset I
\]
such that
\begin{equation}\label{eq:one-scale-ass-en}
\Biggl(\sum_{j\ge k_*-C_0}
\norm{P_j\nabla u}_{L^4_{t,x}(J_*\times\R^2)}^2\Biggr)^{1/2}
\le \eps_*,
\end{equation}
then $u$ can be uniquely extended to
\[
u\in L_t^\infty\bigl([t_0,t_*+T_*]:\dot H^1\cap\dot H^3\bigr).
\]
\item
Let $u$ be the maximal-lifespan solution of \eqref{eq:SM-en}, with lifespan interval $I_{\max}=[0,T_+)$, and suppose $T_+<\infty$. Then there exists $\rho=T_*/2>0$ such that for all $t\in(T_+-\rho,T_+)$ and all $k$ satisfying $J_{t,k}=[t-2^{-2k},t]\subset[0,T_+)$, one has
\[
\Biggl(\sum_{j\ge k-C_0}
\norm{P_j\nabla u}_{L^4_{t,x}(J_{t,k}\times\R^2)}^2\Biggr)^{1/2}
\ge \eps_*.
\]
\end{enumerate}
\end{proposition}

The key point in part \textup{(i)} of Proposition \ref{prop:one-scale-en} is that, on a dyadic time window of length $2^{-2k_*}$, the low-frequency part is automatically small enough.

\begin{lemma}[Automatic low-frequency smallness]\label{lem:low-auto-en}
Let $J_*=[t_*-2^{-2k_*},t_*]$. Then, for every integer $C_0\ge10$,
\begin{equation}\label{eq:lowauto-en}
\sum_{j\le k_*-C_0}
\norm{P_j\nabla u}_{L_{t,x}^4(J_*\times\R^2)}^2
\lesssim 2^{-C_0}E(u).
\end{equation}
\end{lemma}

\begin{proof}
By H\"older and Bernstein,
\[
\norm{P_j\nabla u}_{L_{t,x}^4(J_*\times\R^2)}
\le |J_*|^{1/4}\norm{P_j\nabla u}_{L_t^\infty L_x^4}
\lesssim 2^{-k_*/2}2^{j/2}\norm{P_j\nabla u}_{L_t^\infty L_x^2}.
\]
Squaring and summing, we obtain
\[
\sum_{j\le k_*-C_0}
\norm{P_j\nabla u}_{L_{t,x}^4(J_*\times\R^2)}^2
\lesssim
\sum_{j\le k_*-C_0}2^{j-k_*}\norm{P_j\nabla u}_{L_t^\infty L_x^2}^2
\lesssim
2^{-C_0}\sum_j\norm{P_j\nabla u}_{L_t^\infty L_x^2}^2,
\]
and the last quantity is controlled by energy conservation.
\end{proof}

\begin{proof}[Proof of Proposition \ref{prop:one-scale-en}\textup{(i)}]
Fix a constant $C_0\ge20$ to be chosen. By the estimate \eqref{eq:lowauto-en} in Lemma \ref{lem:low-auto-en},
\[
\sum_{j<k_*-C_0}
\norm{P_j\nabla u}_{L_{t,x}^4(J_*\times\R^2)}^2
\lesssim 2^{-C_0}E_0.
\]
Choosing $C_0$ sufficiently large makes this low-frequency part smaller than $\eps_0^2/4$. On the other hand, assumption \eqref{eq:one-scale-ass-en} says that the high-frequency tail is smaller than $\eps_*^2$. If $\eps_*$ is chosen sufficiently small so that
\[
\eps_*^2+C(E_0)^{-2}\cdot 2^{-C_0}E_0\le \frac14\eps_0^2,
\]
and if we use that the number of intermediate frequency layers is only $O(C_0)$, then
\[
\norm{\nabla u}_{\ell^2L_{t,x}^4(J_*\times\R^2)}
\le \eps_0.
\]
Proposition \ref{prop:cont-step-en} can therefore be applied on the interval $J_*$, and it uniquely extends the solution to $t_*+T_*$, where $T_*=T(E_0)$.
\end{proof}

\begin{proof}[Proof of Proposition \ref{prop:one-scale-en}\textup{(ii)}]
Let $\rho=T_*/2$. If the conclusion failed, then there would exist $t\in(T_+-\rho,T_+)$ and $k$ such that $J_{t,k}\subset[0,T_+)$ and
\[
\Biggl(\sum_{j\ge k-C_0}
\norm{P_j\nabla u}_{L^4_{t,x}(J_{t,k}\times\R^2)}^2\Biggr)^{1/2}
<\eps_*.
\]
Applying part \textup{(i)} of Proposition \ref{prop:one-scale-en} to the interval $[0,t]$ and the window $J_{t,k}$, we extend the solution to $t+T_*$. Since $t>T_+-T_*/2$, we have $t+T_*>T_+$, contradicting maximality of the lifespan.
\end{proof}

\begin{proof}[Proof of Theorem \ref{thm:window-global-en}]
Take $C_0,\eps_*,T_*$ from Proposition \ref{prop:one-scale-en}, and set $\rho:=T_*/2$. If $T_+<\infty$ and \eqref{eq:window-global-ass-en} holds, then one may choose $t_*\in(T_+-\rho,T_+)$ and $k_*\in\mathbb Z$ such that $J_{t_*,k_*}\subset[0,T_+)$ and
\[
\Biggl(\sum_{j\ge k_*-C_0}
\norm{P_j\nabla u}_{L^4_{t,x}(J_{t_*,k_*}\times\R^2)}^2\Biggr)^{1/2}
<\eps_*.
\]
Applying part \textup{(i)} of Proposition \ref{prop:one-scale-en} to the interval $[0,t_*]$, we obtain an extension of the solution to $t_*+T_*$. Since $t_*>T_+-T_*/2$, one has $t_*+T_*>T_+$, contradicting the maximality of $T_+$. Thus any maximal solution satisfying \eqref{eq:window-global-ass-en} must have $T_+=+\infty$.

The equivalent statement \eqref{eq:late-window-lower-en} is exactly the same conclusion as part \textup{(ii)} of Proposition \ref{prop:one-scale-en}, with $\rho=T_*/2$ inserted into the theorem statement in Section 1.
\end{proof}

\end{document}